\documentclass[letterpaper,10pt,twocolumn]{scrartcl}

\usepackage[top=18mm,bottom=20mm,left=18mm,right=18mm,columnsep=15pt]{geometry} 
\usepackage[format=plain,small,labelfont=bf,up,textfont=it,up]{caption} 
\usepackage{booktabs} 
\usepackage{type1cm}
\usepackage{lettrine} 
\usepackage{abstract} 
\usepackage{titlesec} 
\renewcommand\thesection{\Roman{section}} 
\renewcommand\thesubsection{\Alph{subsection}} 
\titleformat{\section}[block]{\scshape\centering}{\thesection.}{1em}{} 
\titleformat{\subsection}[block]{\it}{\thesubsection.}{1em}{} 

\usepackage{cite}

\usepackage[dvips]{graphicx}
\usepackage[usenames,dvipsnames]{xcolor}
\usepackage{tikz}
\usetikzlibrary{shapes,arrows,trees}
\usepackage{psfrag}
\usepackage{booktabs}

\usepackage[cmex10]{amsmath}
\usepackage{amssymb,amsfonts,amsthm}
\usepackage{mathrsfs}
\usepackage{dsfont}
\usepackage{cases}

\usepackage{algorithmic}

 \usepackage[caption=false,font=footnotesize]{subfig}

\usepackage{fixltx2e}

\usepackage{url}

\usepackage[square,sort,comma,numbers]{natbib}
\newtheorem{defn}{Definition}

\newtheorem{cor}{Corollary}
\newtheorem{lem}{Lemma}
\newtheorem{thm}{Theorem}

\newtheorem{ass}{Assumption}
\newtheorem{rem}{Remark}

\definecolor{darkgreen}{rgb}{0,0.4,0}
\definecolor{darkgray}{rgb}{0.3,0.3,0.3}

\newcommand\TRANSP{{\top}}
\newcommand\pplus{+\!\!+}

\title{\fontsize{22pt}{16pt}\selectfont{\bf\normalfont Relaxed Logarithmic Barrier Function Based\\ Model Predictive Control of Linear Systems}} 

\author{Christian~Feller~and~Christian~Ebenbauer %
\thanks{This work was supported by the Deutsche Forschungsgemeinschaft
(Emmy-Noether-Grant, Novel Ways in Control and Computation, EB 425/2-1, and Cluster of Excellence in Simulation Technology,
EXC 310/2).}\\[2mm] 
\normalsize Institute for Systems Theory and Automatic Control \\
\normalsize University of Stuttgart, Pfaffenwaldring 9, 70550 Stuttgart, Germany \\ 
\normalsize e-mail:\,\{christian.feller,ce\}@ist.uni-stuttgart.de 
\vspace{-5mm}
}
\date{}

\begin{document}
\twocolumn[
  \begin{@twocolumnfalse}
    \maketitle
\vspace*{-0.25cm}
\begin{abstract}
In this paper, we investigate the use of relaxed logarithmic barrier functions in the context of linear model predictive control. 
We present results that allow to guarantee asymptotic stability of the corresponding closed-loop system, and discuss further properties like performance and constraint satisfaction in dependence of the underlying relaxation. 
The proposed stabilizing MPC schemes are not necessarily based on an explicit terminal set or state constraint and allow to characterize the stabilizing control input sequence as the minimizer of a globally defined, continuously differentiable, and strongly convex function. The results are illustrated by means of a numerical example.
\end{abstract}
\vspace*{0.6cm}
  \end{@twocolumnfalse}
]
\saythanks
%
%
\section{Introduction}
%
\lettrine[nindent=0em,lines=2]{T}he basic idea of model predictive control (MPC) is to obtain a stabilizing feedback law by solving at each sampling instant, in a receding horizon fashion, a suitable finite-horizon open-loop optimal control problem that is parametrized by the evolving system state.
Only the first element of the optimal control input is applied to the plant and the optimization is repeated at the next time step in a receding horizon fashion.
Based on the predicted dynamics of the system to be controlled, both a user-defined cost objective and potential constraints on the system states and input can be taken into account explicitly within the corresponding optimization problem.
There exist various theoretical results concerning the stability properties of the closed-loop system for both linear and nonlinear systems, and MPC can be considered as a widely accepted control concept that is more and more applied to industrial processes, see~\cite{mayne00} and \cite{qin03} as well as references therein. \\
However, there often is a gap between theoretical MPC concepts and their actual implementation, and several disadvantages remain.
For example, algorithmic aspects of the required optimization are usually not taken into account explicitly in the MPC design, which may result in suboptimal or even unstable closed-loop behavior.  
Furthermore, in the presence of disturbances, sensor outliers, or observer errors, the underlying open-loop optimal control problem may become infeasible, leading to a complete crash of the respective control algorithm.
In recent years, considerable research effort has been spent on designing efficient and reliable algorithmic MPC implementations that, if possible, still preserve some of the desired theoretical properties when running in closed-loop operation, see e.g. \cite{bemporadExplicitLQR,diehl05,zeilinger11,patrinos12,borrelliBook14}.  
One particular approach that allows to reduce the gap between MPC schemes on the one and MPC algorithms on the other hand is given by the concept of barrier function based MPC, which was introduced in~\cite{wills04} and has been extended recently in~\cite{feller13,feller14b,feller14c}. In barrier function based MPC schemes, the inequality constraints occurring in the open-loop optimal control problem are incorporated into the cost function by means of suitable barrier function terms like it is also done in interior-point optimization algorithms~\cite{nesterovBook}. As shown in the above references, barrier function based MPC approaches allow to guarantee asymptotic stability of the closed-loop system 
while at the same time reducing the underlying open-loop optimal control problem to an equality constrained or even unconstrained optimization problem. Thus, a barrier function based reformulation is already integrated into the MPC design, which then allows for an efficient implementation based on tailored numerical optimization techniques.
In particular, the barrier function based approach makes the open-loop optimal control problem accessible for continuous-time algorithms that asymptotically track the optimal solution of the unconstrained reformulation and, thus, allow to implement MPC without any iterative on-line optimization, see \cite{ohtsuka04,deHaan07,feller13,feller14}. 
However, two main disadvantages remain. On the one hand, the open-loop optimal control problem may still become infeasible since the barrier functions are only defined within the interior of the corresponding constraint sets. On the other hand, the barrier function based approach inherently requires the use of a terminal set or terminal equality constraint, which is not desirable, and also typically not used, in practice.\\[0.2cm]
In this paper, we show that both of these problems can be eliminated by 
making use of so-called relaxed logarithmic barrier functions, i.e., barrier functions that are smoothly extended by a suitable penalty term outside of the corresponding constraint set~\cite{benTal92,nash94,hauser06}.
We discuss suitable relaxation procedures (Section~\ref{sec:relBarrierIntro}) and present results that allow to guarantee asymptotic stability of the closed-loop system with and without making use of terminal sets (Sections \ref{sec:relBarrierStab1}, \ref{sec:relBarrierStab2}), to determine and control the maximal violation of input and state constraints in closed-loop operation~(Section~\ref{sec:relBarrierConstr}), and to recover the optimality properties of both nonrelaxed barrier function based and conventional linear MPC schemes~(Section~\ref{sec:relBarrierOpt}).
Furthermore, based on the presented results, we provide in Section~\ref{sec:relBarrierNumExample} a step-by-step procedure for the constructive design of the overall MPC scheme. \\
The basic idea of relaxed barrier function based MPC is closely related to approaches based on soft constraints or penalty functions, see e.g.~\cite{kerrigan00,zeilinger10}. However, in these approaches, closed-loop properties like stability or strict constraint satisfaction can usually only be guaranteed when making use of nonsmooth or even exact penalty functions.
The key feature of the relaxed barrier functions discussed in this paper is that they result in a smooth formulation of the overall problem while still allowing for an arbitrary close approximation of the corresponding nonrelaxed case. 
As a result, the stabilizing control input can be characterized as the minimizer of a globally defined, twice continuously differentiable, and strongly convex cost function, which makes it efficiently computable by standard nonlinear programming algorithms.
Some preliminary results on the concept of relaxed barrier function based MPC have been presented in~\cite{feller14}.  
In this paper we significantly extend these results and provide an in-depth study of several theoretical and practical aspects of the presented linear MPC schemes.
A particularly interesting new result is the insight that the use of relaxed barrier functions allows to ensure global asymptotic stability of the closed-loop system without making use of a terminal set.\\[0.3cm]
Throughout the paper we will make use of the following notation. $\mathbb{R}_+$, $\mathbb{R}_{+\!\!+}$, and $\mathbb{N}_+$ denote the sets of nonnegative real, strictly positive real, and strictly positive natural numbers. Furthermore, $\mathbb{S}^n_+$ and $\mathbb{S}^n_{\pplus}$ refer to the sets of positive semi-definite and positive definite matrices of dimension $n\in\mathbb{N}_+$. For any given matrix~$M$ or vector $v$, $M^i$ and $v^i$ refer to the $i$-th row or element and $\lVert x \rVert_M:=\sqrt{x^\TRANSP \! M x}$ for any $M\in\mathbb{S}_+$. A polytope is defined as the compact intersection of a finite number of halfspaces. For any arbitrary set $S$, the expression $S^\circ$ will denote the open interior and $\partial S$ the boundary. Moreover, $\mathds{1}:=\begin{bmatrix} 1 & \cdots & 1 \end{bmatrix}^\TRANSP$ of suitable dimension. 
\section{Problem Setup}
In this paper, we consider the control of linear time-invariant discrete-time systems of the form
\begin{equation}\label{eq:discreteSystem}
x(k+1)=Ax(k)+Bu(k)\, ,
\end{equation}
where $x(k) \in \mathbb{R}^n$ refers to the vector of system states and $u(k) \in \mathbb{R}^m$ refers to the vector of system inputs, both at time instant $k \geq 0$. Moreover, the matrices $A\in \mathbb{R}^{n\times n}$ and $B\in \mathbb{R}^{n\times m}$ describe the corresponding system dynamics, where we assume $(A,B)$ to be stabilizable. The control task is to regulate the system state to the origin while minimizing a given, user defined performance criterion and satisfying state and input constraints of the form $x(k) \in \mathcal{X}$ and $u(k) \in \mathcal{U}$ for all $k \geq 0$. Here, $\mathcal{X}\!\subset \mathbb{R}^n$ and $\,\mathcal{U}\!\subset \mathbb{R}^m$ are typically considered to be given convex sets that contain the origin in their interior. In this paper, we assume $\mathcal{X}$ and $\mathcal{U}$ to be polytopes defined as 
\begin{subequations}%
\begin{align}
 \mathcal{X}&=\{x \in \mathbb{R}^n: C_{\mathrm{x}} x \leq d_{\mathrm{x}}\}\, , \\ \mathcal{U}&=\{u \in \mathbb{R}^m: C_{\mathrm{u}} u \leq d_{\mathrm{u}}\} \, ,
\end{align} 
\end{subequations}%
where $C_{\mathrm{x}}\in\mathbb{R}^{q_{\mathrm{x}}\times n}$, $C_{\mathrm{u}}\in\mathbb{R}^{q_{\mathrm{u}}\times m}$ and $d_{\mathrm{x}}\in\mathbb{R}^{q_{\mathrm{x}}}_{+\!\!+}$, $d_{\mathrm{u}}\in\mathbb{R}^{q_{\mathrm{u}}}_{+\!\!+}$ with $q_x, q_u\in\mathbb{N}_+$.
In linear MPC, this problem setup is usually handled by solving at each sampling instant an open-loop optimal control problem of the form
\begin{subequations} \label{eq:OptProblemLin}
\begin{align}
J_N^{\ast}(x)&=\min_{\boldsymbol{u}} \sum_{k=0}^{N-1} \ell(x_k,u_k) + F(x_N) \\
\text{s.t.\ } \ &{x}_{k+1}=A{x}_{k}+B{u}_{k}, \ {x}_0=x \ , \\ 
&x_k \in \mathcal{X}, \ k=0,\dots, N-1, \ x_N \in \mathcal{X}_f \ ,\\
&u_k \in \mathcal{U}, \ k=0,\dots, N-1 \ .
\end{align}
\end{subequations}
for the current system state~$x=x(k)$ and a finite prediction and control horizon~$N \in\mathbb{N}_+$.
Here, the stage cost $\ell:\mathbb{R}^n\times \mathbb{R}^m\to \mathbb{R}_+$ and the terminal cost $F:\mathbb{R}^n\to\mathbb{R}_+$ are usually defined as $\ell(x,u)=\lVert x\rVert_Q^2+\lVert u\rVert_R^2$ and $F(x)=\lVert x\rVert_P^2$ for appropriately chosen weight matrices $Q\in \mathbb{S}^n_+$, $R \in \mathbb{S}^m_{\pplus}$, $P\in \mathbb{S}^n_{\pplus}$. Furthermore, $\boldsymbol{u}=\{u_0,\dots,u_{N-1}\}$ denotes the sequence of control inputs over the prediction horizon~$N$, while $\mathcal{X}_f$ refers to a closed and convex terminal constraint set that may be used to guarantee stability properties of the closed-loop system. 
Note that we make use of subindices to distinguish open-loop predictions~$x_k$,~$u_k$~from actual state and input trajectories~$x(k), u(k)$.
The control law is obtained by solving~(\ref{eq:OptProblemLin}) at each sampling instant $k\geq 0$ and applying  $u(k)=u_0^*(x(k))$ in a receding horizon fashion.
Sufficient conditions for the recursive feasibility of~(\ref{eq:OptProblemLin}) as well as for the asymptotic stability of the closed-loop system are summarized in~\cite{mayne00}.
In the following, we will refer to this setup as \emph{conventional} linear MPC. 
Moreover, we will often use $x^+$ to denote the successor state and write~(\ref{eq:discreteSystem}) as $x^+=Ax+Bu$. For given system state $x=x(k)$ and input sequence $\boldsymbol{u}=\{u_0,\dots,u_{N-1}\}$, the resulting open-loop state sequence is given by $\boldsymbol{x}(\boldsymbol{u},x)=\{x_0(\boldsymbol{u},x),\dots,x_N(\boldsymbol{u},x)\}$, where the elements $x_k(\boldsymbol{u},x)$, $k=0,\dots,N$ are given according to~(\ref{eq:OptProblemLin}b) with $x_0(\boldsymbol{u},x)=x$. For a given optimal input sequence $\boldsymbol{u}^*(x)$, we write $\boldsymbol{x}^*(x):=\boldsymbol{x}(\boldsymbol{u}^*(x),x)$ as well as $x_k^*(x):=x_k(\boldsymbol{u}^*(x),x)$. Sometimes we also drop the explicit dependence on the current system state for ease of notation.

\section{Preliminary Results}
We introduce in this section some concepts and results on nonrelaxed logarithmic barrier function based MPC that we will need in the remainder of the paper. In particular, we present a novel stability theorem which generalizes different existing ideas and is applicable to a wide class of barrier function based MPC approaches.
\subsection{Barrier Function Based Model Predictive Control}\label{sec:barrierStab}
The main idea in barrier function based MPC is to eliminate the inequality constraints from the above MPC open-loop optimal control problem by making use of suitable barrier functions with a corresponding weighting factor. Based on this idea, it is possible to reformulate problem~(\ref{eq:OptProblemLin}) as an equality constrained (or even unconstrained) strongly convex optimization problem, which can then be solved by means of tailored optimization procedures like, e.g., the Newton-method. In general, the exact solution to the original problem, and hence also the stability properties of the corresponding closed-loop system, are recovered when the weighting factor of the barrier functions approaches zero~\cite[ch. 11]{boydBook}. However, for an arbitrary but fixed nonzero weighting, as it will necessarily occur in all numerical implementations, stability of the origin is by no means guaranteed. Several approaches towards the stabilizing design of barrier function based MPC schemes have been presented in~\cite{wills04,feller13,feller14b,feller14c}. In the following, we shortly summarize the main aspects of barrier function based linear MPC and present a fairly general stability theorem for the considered problem setup with polytopic input and state constraints.
Details on the discussed MPC approaches can be found in the respective references. \\[0.2cm]
Let us consider in the following the barrier function based open-loop optimal control problem
\begin{subequations}\label{eq:OptProblemBarrier}
\begin{align}
\tilde{J}_N^{\ast}(x)&=\min_{\boldsymbol{u}}\sum_{k=0}^{N-1} \tilde{\ell}(x_k,u_k) + \tilde{F}(x_N)   \\
\mbox{s.\,t.} \ \ & {x}_{k+1}=A{x}_{k}+B{u}_{k}, \ {x}_0=x\,, 
\end{align}
\end{subequations}
where $ \tilde{\ell}(x,u):= \ell(x,u)+ \varepsilon B_{\mathrm{u}}(u)+\varepsilon B_{\mathrm{x}}(x)$ and $\tilde{F}(x):=F(x)+\varepsilon B_{\mathrm{f}}(x)$ are the modified stage and terminal cost terms for $\ell(x,u)=\|x\|_Q^2+\|u\|_R^2$ and $F(x)=\|x\|_P^2$, and $B_{\mathrm{u}}(\cdot)$, $B_{\mathrm{x}}(\cdot)$, and $B_{\mathrm{f}}(\cdot)$ are suitable convex barrier functions with domains $\,\mathcal{U}^\circ$, $\mathcal{X}^\circ$, and $\mathcal{X}_f^\circ$ with $B_{\mathrm{u}}(u) \to \infty$ for $u\to \partial \mathcal{U}$, $B_{\mathrm{x}}(x) \to \infty$ for $x\to \partial \mathcal{X}$, and $B_{\mathrm{f}}(x) \to \infty$ for $x\to \partial \mathcal{X}_f$. Furthermore, the positive scalar $\varepsilon >0$ is the barrier function weighting parameter which determines the influence of the barrier function values on the overall cost function. 
As outlined above, the goal is now to choose the problem parameters $P$ and $\mathcal{X}_f$ as well as convex barrier functions $B_{\mathrm{u}}(\cdot)$, $B_{\mathrm{x}}(\cdot)$, and $B_{\mathrm{f}}(\cdot)$ in such a way that a linear MPC scheme based on~(\ref{eq:OptProblemBarrier}) asymptotically stabilizes the origin for any arbitrary but fixed $\varepsilon \in \mathbb{R}_{+\!\!+}$.
If we want to employ standard MPC stability concepts which are based on using the value function $\tilde{J}_N^*(\cdot)$ as a Lyapunov function for the closed-loop system, this usually requires the overall cost function to be continuous as well as positive definite with respect to the origin. 
In order to ensure the latter property for general barrier functions, the concept of gradient recentered barrier functions has been introduced in~\cite{wills04}. In addition, a weighting based recentering approach for logarithmic barrier functions on polytopic sets has been proposed in~\cite{feller14c}. \vspace*{-0.2cm}
\begin{defn}[Recentered log barrier function~\cite{wills04,feller14c}] \label{def:recBarrier}
Let $\mathcal{P}=\{z\in\mathbb{R}^r: Cz\leq d\}$ with $C\in\mathbb{R}^{q\times r}$, $d\in \mathbb{R}^q_{+\!\!+}$ be a polytopic set containing the origin and let $\bar{B}: \mathcal{P}^\circ\to \mathbb{R}$, $\bar{B}(z)=\sum_{i=1}^q \bar{B}_i(z)$ with $\bar{B}_i(z)=-\ln(-C^iz+d^i)$ be the corresponding logarithmic barrier function. 
Then, the function ${B}_G:\mathcal{P}^\circ \to\mathbb{R}_+$ with 
\begin{equation}
{B}_G(z)=\bar{B}(z)-\bar{B}(0)-[\nabla \bar{B}(0)]^\TRANSP z
\end{equation}
defines the gradient recentered logarithmic barrier function for the polytopic set $\mathcal{P}$. Furthermore, ${B}_W:\mathcal{P}^\circ\to\mathbb{R}_+$ with 
\begin{equation}
{B}_W(z)=\sum_{i=1}^q(1+w^i)\left(\bar{B}_i(z)-\bar{B}_i(0) \right)
\end{equation}
defines the corresponding weight recentered logarithmic barrier function, where the weighting vector $w \in \mathbb{R}^q_+$ is chosen in such a way that\! $\nabla {B}_W(0)=\sum_{i=1}^q(1+w^i)\nabla \bar{B}_i(0)=0$.
\end{defn}
In principle, the recentering can be seen as a modification which preserves the main characteristics of the barrier function while ensuring that it is positive definite with respect to the origin.
Both recentering approaches can in general be applied to our problem setup and will result in a continuously differentiable cost function that is positive definite and strongly convex in the optimization variable~$\boldsymbol{u}$. Moreover, the stability results presented in this work do hold independently of the underlying recentering method.\\ 
Let us now turn towards the stability properties of the closed-loop system when applying the barrier function based MPC feedback $u(k)=\tilde{u}(x(k))$. 
\begin{defn} \label{def:feasSet}
For $N\in \mathbb{N}_+$, let the feasible set $\mathcal{X}_N$ be defined as 
$\mathcal{X}_N:=\{x \in\mathcal{X}: \exists\, \boldsymbol{u}=\{u_0,\dots,u_{N-1}\}$ such that $u_k \in \mathcal{U}$, $ x_k(\boldsymbol{u},x) \in \mathcal{X}$ for $k=0,\dots,N-1$ and $x_N(\boldsymbol{u},x) \in \mathcal{X}_f\}$.
\end{defn}
\begin{defn}
 In the following, the matrix $A_K:=A+BK$ describes the closed-loop dynamics for a given stabilizing local control law $u=Kx$ and $B_K(x):=B_{\mathrm{x}}(x)+B_{\mathrm{u}}(Kx)$ refers to the corresponding combined barrier function of input and state constraints for the set $\mathcal{X}_K:=\{x\in \mathcal{X}: Kx \in \mathcal{U}\}$.
\end{defn}
As outlined above, different approaches towards the stabilizing design of barrier function based MPC formulations, i.e., on how to choose the terminal cost matrix $P$, the terminal set $\mathcal{X}_f$, and the corresponding barrier function $B_{\mathrm{f}}(\cdot)$, have been presented in~\cite{wills04,feller13,feller14b} and \cite{feller14c}. In principle, all mentioned approaches are based on the idea of choosing the terminal set $\mathcal{X}_f$ as a positively invariant subset of the state space in which the function $B_K(\cdot)$, i.e., the influence of the input and state constraint barrier functions, can be upper bounded by a quadratic function. Then, based on this quadratic bound, the terminal cost matrix $P$ is computed in such a way that it compensates for this influence and ensures that the barrier function based terminal cost $\tilde{F}(\cdot)$ is a local control Lyapunov function for the auxiliary control law $u=Kx$. While the approaches in \cite{wills04} 
and \cite{feller13} make use of ellipsoidal terminal sets, the approaches presented in \cite{feller14b} and \cite{feller14c} are based on polytopic terminal sets. 
In the following, we present a set of sufficient stability conditions that is used in the remainder of this work and can be seen as a generalization of the ideas presented in the above references. \vspace*{-0.15cm}
 \begin{ass}\label{ass:barrierMPCStab}
 For $Q\in\mathbb{S}^n_{+}$, $R\in\mathbb{S}^m_{+\!\!+}$, $\varepsilon \in \mathbb{R}_{+\!\!+}$ and a given stabilizing local control gain $K\in\mathbb{R}^{m\times n}$ let the parameters of the barrier function based open-loop optimal control problem~(\ref{eq:OptProblemBarrier}) satisfy the following conditions.
 \begin{itemize}\setlength{\itemsep}{0pt}
 \item[\textbf{A1}:] The barrier functions $B_{\mathrm{u}}(\cdot)$ and $B_{\mathrm{x}}(\cdot)$ are recentered barrier functions for the sets $\mathcal{U}$ and $\mathcal{X}$, respectively.
 \item[\textbf{A2}:] There exists $M \in \mathbb{S}^{n}_+$, s.\,t. $B_K(x)\leq x^{\TRANSP}\!Mx \ \forall x \in \mathcal{N}$, where $\mathcal{N}\subset \mathcal{X}_K$, $0 \in \mathcal{N}^\circ$, is a convex, compact set.
 \item[\textbf{A3}:] The matrix $P\in\mathbb{S}^n_{+\!\!+}$ is a solution to the Lyapunov equation $P=A_K^{\TRANSP}PA_K^{}+K^{\!\TRANSP}\!RK+Q+\varepsilon M$.
  \item[\textbf{A4}:] The terminal set $\mathcal{X}_f$ is convex and compact with $0\in\mathcal{X}_f$, $\mathcal{X}_f \subset \mathcal{N}\subset \mathcal{X}_K$, and $x^+\!=A_Kx\in\mathcal{X}_f \ \forall x \in \mathcal{X}_f$.
 \item[\textbf{A5}:] The function $B_{\mathrm{f}}(\cdot)$ is a recentered barrier function  for the set $\mathcal{X}_f$ and $B_{\mathrm{f}}(A_Kx)-B_{\mathrm{f}}(x)\leq 0 $ $ \ \forall x\in \mathcal{X}_f^\circ$.
 \end{itemize}
 \end{ass}
Based on Assumption~\ref{ass:barrierMPCStab}, we can state the following result on the stability of the closed-loop system.
 \begin{thm}\label{thm:barrierMPCStab}
 Let Assumption~\ref{ass:barrierMPCStab} hold true and let the feasible set $\mathcal{X}_N$ defined according to Definition~\ref{def:feasSet} have nonempty interior. Then, the barrier function based MPC feedback $u(k)=\tilde{u}_0^*(x(k))$ resulting from problem~(\ref{eq:OptProblemBarrier}) asymptotically stabilizes the origin of system~(\ref{eq:discreteSystem}) under strict satisfaction of all input and state constraints for any initial condition $x(0)\in \mathcal{X}_N^\circ$.
 \end{thm}
\begin{proof}
 The proof uses standard MPC stability arguments and comprises some of the main ideas presented in~\cite{wills04,feller13} and \cite{feller14b}. 
We first show recursive feasibility of problem~(\ref{eq:OptProblemBarrier}). For any $x_0 \in \mathcal{X}_N^\circ$, there exists by definition an optimal input sequence $\tilde{\boldsymbol{u}}^*(x_0)=\{\tilde{u}_0^*, \dots, \tilde{u}_{N-1}^*\}$ that guarantees strict satisfaction of all input, state, and terminal set constraints and results in a feasible open-loop state sequence $\boldsymbol{x}^*(x_0)=\{x_0,x_1^*,\dots,x_N^*\}$ with $x_N^*\in \mathcal{X}_f^\circ$. The successor state is given as $x_0^+=Ax_0+B\tilde{u}_0^*=x_1^*$. Due to the linear dynamics and the properties of the terminal set $\mathcal{X}_f$, see A4 of Assumption~\ref{ass:barrierMPCStab}, $\tilde{\boldsymbol{u}}^+(x_0)=\{\tilde{u}_1^*, \dots, \tilde{u}_{N-1}^*, Kx_N^*\}$ is a suboptimal but feasible input sequence for the initial state $x_0^+$ that results in the in the feasible open-loop state sequence $\boldsymbol{x}(\tilde{\boldsymbol{u}}^+(x_0),x_0^+)=\{x_0^+,x_2^*,\dots,x_N^*,A_Kx_N^*\}$ 
with $A_Kx_N^* \in \mathcal{X}_f^\circ$. This shows that for any $x_0 \in \mathcal{X}_N^\circ$, there exists a feasible input sequence that ensures that the successor state $x_0^+=Ax_0+B\tilde{u}_0^*$ lies again in the interior of the feasible set~$\mathcal{X}_N$, which guarantees recursive feasibility of the open-loop optimal control problem~(\ref{eq:OptProblemBarrier}).\\ In the following, we show that the value function satisfies
\begin{equation}    \label{eq:jdrecease}
\tilde{J}_N^*(x_0^+)-\tilde{J}_N^*(x_0)\leq  -\tilde{\ell}(x_0,\tilde{u}_0^*(x_0)) \ \ \forall x_0 \in \mathcal{X}_N^\circ \ .
\end{equation} 
First, due to the suboptimality of~$\tilde{\boldsymbol{u}}^+(x_0)$, it holds that $\tilde{J}_N^*(x_0^+)-\tilde{J}_N^*(x_0)\leq \tilde{J}_N(\tilde{\boldsymbol{u}}^+(x_0),x_0^+)-\tilde{J}_N^*(x_0)$, where $\tilde{J}_N(\tilde{\boldsymbol{u}}^+(x_0),x_0^+)$ denotes the value of the cost function evaluated for the suboptimal input sequence~$\tilde{\boldsymbol{u}}^+(x_0)$. Moreover, $\tilde{J}_N(\tilde{\boldsymbol{u}}^+(x_0),x_0^+)-\tilde{J}_N^*(x_0) = \tilde{F}(A_Kx_N^*) -\tilde{F}(x_N^*) + \tilde{\ell}(x_N^*,Kx_N^*) - \tilde{\ell}(x_0,\tilde{u}_0^*) \leq -\tilde{\ell}(x_0,\tilde{u}_0^*)$ for any $x_0 \in \mathcal{X}_N^\circ$ since
\begin{subequations}
\begin{align}
&\tilde{F}(A_Kx_N^{*}) -\tilde{F}(x_N^*) + \tilde{\ell}(x_N^*,Kx_N^*)  \\
\ &= \lVert A_K x_N^* \rVert^2_{P} - \lVert x_N^*\rVert^2_P  + \lVert x_N^* \rVert^2_Q + \lVert Kx_N^* \rVert^2_{R}  \\
\ & \quad + \varepsilon B_K(x_N^*) + \varepsilon B_{\mathrm{f}}(A_Kx_N^*)- \varepsilon B_{\mathrm{f}}(x_N^*) \\
\ & \leq \varepsilon \left(B_{\mathrm{f}}(A_Kx_N^*)- B_{\mathrm{f}}(x_N^*)\right) \leq 0 \ \ \forall \ x_N^* \in \mathcal{X}_f^\circ \ .
\end{align}
\end{subequations}
Here, the first inequality follows from the quadratic bound $B_K(x)\leq \gamma\, x^\TRANSP\! Mx \ \forall x \in \mathcal{X}_f \subset  \mathcal{N}$ and the suitable choice of the terminal cost matrix~$P$, see A2 and A3 of Assumption~\ref{ass:barrierMPCStab}. The second inequality holds due to A5 of Assumption~\ref{ass:barrierMPCStab}. \\
Finally, the considered problem setup and the design of the barrier functions according to A1 of Assumption~\ref{ass:barrierMPCStab} ensure that $\tilde{J}_N^*: \mathcal{X}_N^\circ \to \mathbb{R}_+$ is a well-defined and positive definite function with $\tilde{J}_N^*(x) \to \infty$ whenever $x\to \partial \mathcal{X}_N$. Thus, in combination with the decrease property~(\ref{eq:jdrecease}), it can be used as a Lyapunov function which proves asymptotic stability of the origin of system~(\ref{eq:discreteSystem}) under the feedback $u(k)=\tilde{u}_0^*(x(k))$ for any $x(0) \in \mathcal{X}_N^\circ$.
\end{proof}
 \begin{rem}
  Note that the barrier function parameter $\varepsilon\in \mathbb{R}_{\pplus}$ can in principle be decreased iteratively in each sampling step or between two consecutive sampling steps without loosing feasibility or stability properties of the closed-loop system. This might be meaningful for numerical reasons or in order to enforce convergence to the optimal solution of the conventional MPC problem~(\ref{eq:OptProblemLin}). We limit ourselves to fixed values of $\varepsilon$. 
  \end{rem}
Different approaches on how to actually construct the neighborhood~$\mathcal{N}$ and the corresponding quadratic bound for $B_K(\cdot)$ as well as suitable choices for the terminal set $\mathcal{X}_f$ and the barrier function $B_{\mathrm{f}}(\cdot)$ have been proposed in~\cite{wills04,feller13,feller14b,feller14c}. 
However, the fact that the underlying barrier functions are only defined in the interior of the respective constraint sets may be problematic both from a practical and from a conceptual point of view. On the one hand, violations of the state, input, and terminal set constraints are not tolerated at all, which might cause severe problems in the presence of uncertainties, disturbances, noise, observer errors, or sensor outliers. On the other hand, all existing stability concepts inherently require the use of a suitable terminal set as they are based on upper bounding the barrier function $B_K(\cdot)$ by a quadratic function, which is of course only possible locally in a region around the origin. 
In the following section, we introduce the concept of relaxed logarithmic barrier function based MPC and show that it can be used to overcome all these limitations, allowing for conceptually simpler and more reliable linear MPC schemes. 
%
\section{Main Results}
The basic idea of relaxed logarithmic barrier functions is to smoothly extend a given logarithmic barrier function with a suitable, globally defined penalty function~\cite{benTal92,nash94,hauser06}.
In the following, we provide an in-depth study of several interesting theoretical and practical aspects of linear MPC approaches that are based on such relaxed logarithmic barrier functions.
In particular, we show that, on the one hand, feasibility and stability properties of the nonrelaxed formulation can always be recovered by approximating the original barrier functions close enough.  
On the other hand, we present novel linear MPC schemes that, by exploiting the properties of relaxed logarithmic barrier functions, do not require the use of a terminal set and allow to guarantee global asymptotic stability of the closed-loop system, see Section~\ref{sec:relBarrierStab2}.
The benefits of global stabilization and a simplified design procedure are bought with the loss of guaranteed input and state constraint satisfaction. However, as we will show in Section~\ref{sec:relBarrierConstr}, the relaxed barrier functions can always be designed in such a way that an a-priori defined tolerance for the maximal violation of state and input constraints is guaranteed for a certain set of initial conditions. We briefly discuss some interesting performance and robustness properties of the proposed MPC schemes in Section~\ref{sec:relBarrierOpt}.\vspace*{-0.1cm}
\subsection{Relaxed Logarithmic Barrier Function Based MPC} \label{sec:relBarrierIntro}
We begin our studies by introducing the concept of relaxed logarithmic barrier functions and discussing suitable realizations based on different relaxing functions. \vspace*{-0.1cm}
\begin{defn}[Relaxed logarithmic barrier function] \label{def:relBarrier} \ Consider a scalar $\delta  \in \mathbb{R}_{+\!\!+}$, called the relaxation parameter, and let $\beta(\cdot\,;\delta):(-\infty,\delta]\to \mathbb{R}$ be a strictly monotone and continuously differentiable function that satisfies $\beta(\delta;\delta)=-\ln(\delta)$ as well as $\lim_{z\to -\infty}\beta(z;\delta) = \infty$.
Then, we call $\hat{B}:\mathbb{R}\to \mathbb{R}$ defined as
\begin{equation}
 \hat{B}(z)=\begin{cases}
             -\ln(z) & z>\delta \\
             \phantom{-}\beta(z;\delta) & z \leq \delta
            \end{cases}
\end{equation} 
the relaxed logarithmic barrier function for the set $\mathbb{R}_+$ and refer to the function $\beta(\cdot;\delta)$ as the relaxing function.
\end{defn}
A~graphical illustration of the basic idea is given in Fig.~\ref{fig:recMethods}.
In general, it is advisable to choose the relaxing function $\beta(\cdot\,;\delta)$ as a strictly convex $C^2$ function that smoothly extends the natural logarithm at $z=\delta$. In this case, $\hat{B}(\cdot)$ is a strictly convex function that is twice continuously differentiable and defined on $z\in (-\infty,\infty)$. Note that $\lim_{\delta\to 0}\hat{B}(z) \to B(z)$ for any strictly feasible $z\in\mathbb{R}_{+\!\!+}$, which shows that the nonrelaxed formulation can always be recovered by decreasing the relaxation parameter $\delta$ to zero. Note that we do not indicate the explicit dependence of the relaxed barrier functions on the relaxation parameter~$\delta$ for the sake of notational simplicity. However, we will use $\hat{B}(\cdot)$ to denote the relaxed version of a barrier function based expression~${B}(\cdot)$. \\[0.15cm]
The first ideas on relaxed (or approximate) logarithmic barrier functions with a quadratic relaxing function $\beta(\cdot\,;
\delta)$ seem to have been proposed in~\cite{benTal92} and~\cite{nash94}, respectively.
In \cite{hauser06}, the authors extended the concept to general polynomial penalty terms and applied it in the context of continuous-time trajectory optimization. 
In particular, the authors make use of the polynomial relaxing function 
\begin{equation} \label{eq:relFunHauser}
 \beta_k(z;\delta)=\frac{k-1}{k}\left[\left(\frac{z-k\delta}{(k-1)\delta}\right)^k-1\right]-\ln(\delta) \, ,
\end{equation} 
where $k>1$ is an even integer. It is easy to verify that the function $ \beta_k(\cdot\,;\delta)$ has all the desired properties mentioned above. As reported in~\cite{hauser06}, already a quadratic extension based on $k=2$ seems to work well in practice. \\
In order to avoid large constraint violations, it may be beneficial if the relaxing function increases very rapidly outside the border of the feasible set. As an alternative to the polynomial relaxation above, we proposed in~\cite{feller14} the following exponential relaxing function 
\begin{equation}
 \beta_e(z;\delta)=\exp\left(1-\dfrac{z}{\delta}\right)-1-\ln(\delta) \, ,
\end{equation}
which is an upper bound for the function $ \beta_k(\cdot\,;\delta)$. In fact, using the derivatives $\beta_k'(\cdot;\delta)$, $\beta_e'(\cdot;\delta)$ and the limit representation of the exponential function, it can be shown that $\beta_k(z;\delta)\leq \beta_{k+2}(z;\delta)\leq \beta_e(z;\delta)$ and that $\lim_{k\to\infty}\beta_k(z;\delta)=\beta_e(z;\delta)$ $\forall z\leq \delta$. Furthermore, for $\beta(\cdot\,;\delta)=\beta_k(\cdot\,;\delta)$ and $\beta(\cdot\,;\delta)=\beta_e(\cdot\,;\delta)$ it also holds that $\hat{B}(z)\leq B(z) \ \forall z \in \mathbb{R}_{+\!\!+}$.
We assume in the following that one of the two functions $\beta_k(\cdot\,;\delta)$ and $\beta_e(\cdot\,;\delta)$ is used as relaxing function.\\[0.2cm]
%
Based on these concepts, let us now consider the following relaxed barrier function based MPC formulation  
\begin{subequations}\label{eq:OptProblemRelBarrier}
\begin{align}
\hat{J}_N^{\ast}(x;\delta)&=\min_{\boldsymbol{u}}\sum_{k=0}^{N-1} \hat{\ell}(x_k,u_k) + \hat{F}(x_N)   \\
\mbox{s.\,t.} \ \ & {x}_{k+1}=A{x}_{k}+B{u}_{k}, \ {x}_0=x\, , 
\end{align}
\end{subequations}
where $\hat{\ell}(x,u):= \ell(x,u)+ \varepsilon \hat{B}_{\mathrm{u}}(u)+\varepsilon \hat{B}_{\mathrm{x}}(x)$  for suitable  relaxed recentered logarithmic barrier functions $\hat{B}_{\mathrm{u}}(\cdot)$ and $\hat{B}_{\mathrm{x}}(\cdot)$ as discussed below. The term $\hat{F}(\cdot)$ denotes a suitable relaxed terminal cost function which we do not specify at the moment. As we will see in the Sections~\ref{sec:relBarrierStab1}~and~~\ref{sec:relBarrierStab2}, the choice of $\hat{F}(\cdot)$ is crucial for ensuring stability properties of the closed-loop system. 
\begin{ass}\label{ass:relBarriersXU}
The functions $\hat{B}_{\mathrm{u}}(u)=\sum_{i=1}^{q_{\mathrm{u}}}\hat{B}_{\mathrm{u},i}(u)$ and $\hat{B}_{\mathrm{x}}(x)=\sum_{i=1}^{q_{\mathrm{x}}}\hat{B}_{\mathrm{x},i}(x)$ are relaxed recentered logarithmic barrier functions for the polytopic sets $\mathcal{U}$ and $\mathcal{X}$ and the relaxation parameter $\delta \in \mathbb{R}_{+\!\!+}$ is chosen such that
$\hat{B}_{\mathrm{u}}(0)=\hat{B}_{\mathrm{x}}(0)=0$ and $\nabla\hat{B}_{\mathrm{u}}(0)=0$, $\nabla\hat{B}_{\mathrm{x}}(0)=0$.
\end{ass}
Note that a $\delta$ satisfying Assumption~\ref{ass:relBarriersXU} always exists. In particular, if the relaxation parameter satisfies $0<\delta\leq\min\{d_{\mathrm{x}}^1,\dots,d_{\mathrm{x}}^{q_{\mathrm{x}}},d_{\mathrm{u}}^1$, $\dots,d_{\mathrm{u}}^{q_{\mathrm{u}}}\}$, suitable relaxed barrier functions can be easily constructed by simply relaxing the logarithmic terms of the recentered nonrelaxed formulation. Hence, for $z_i(x):=-{C_{\mathrm{x}}^i}x+d_{\mathrm{x}}^i$, $i=1,\dots,q_{\mathrm{x}}$, we get for example 
 \begin{equation}\label{eq:bxGrad}
 \hat{B}_{\mathrm{x},i}(x)\!=\!\begin{cases}
               -\ln(z_i(x))+\ln(d_{\mathrm{x}}^i)-\frac{{C_{\mathrm{x}}^i} x}{d_{\mathrm{x}}^i} & z_i(x) > \delta \\[0.15cm]
              \, \beta(z_i(x);\delta) +\ln(d_{\mathrm{x}}^i)-\frac{{C_{\mathrm{x}}^i} x}{d_{\mathrm{x}}^i} & z_i(x) \leq \delta \,
              \end{cases}
\end{equation} 
when using gradient recentered barrier functions and 
 \begin{equation}\label{eq:bxWeight}
 \hat{B}_{\mathrm{x},i}(x)\!=\!\begin{cases}
               (1+w_{\mathrm{x}}^i)\left(-\ln(z_i(x))+\ln(d_{\mathrm{x}}^i)\right) & z_i(x) > \delta \\[0.1cm]
               (1+w_{\mathrm{x}}^i)\left(\,\beta(z_i(x);\delta) +\ln(d_{\mathrm{x}}^i)\right) & z_i(x) \leq \delta \,
              \end{cases}
\end{equation} 
when using weight recentered barrier functions, cf.~Definition~\ref{def:recBarrier}. 
The barrier functions $\hat{B}_{\mathrm{u},i}(\cdot)$ for the input constraints can be defined analogously.  
 Note that it is in principle possible to consider individual relaxation parameters $\delta_i$ for the different constraints, but we restrict ourselves to one overall $\delta$ for the sake of simplicity.
While the relaxation of the barrier functions directly implies that problem~(\ref{eq:OptProblemRelBarrier}) is always recursively feasible, stability properties of the resulting closed-loop system are by no means guaranteed.
This problem will be addressed in the following two sections. \vspace*{-0.1cm}
 \begin{figure}
\centering
\includegraphics[scale=0.9]{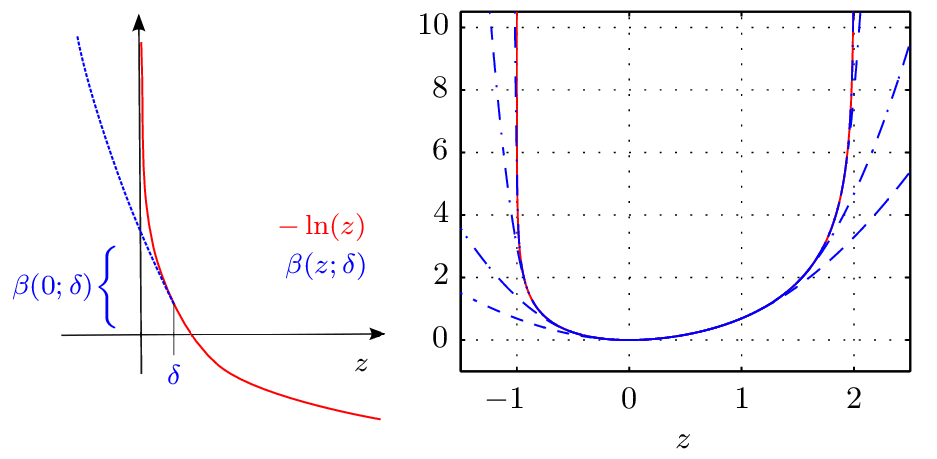}
    \captionof{figure}{Left: Principle of relaxed logarithmic barrier functions based on the relaxing function $\beta(z;\delta)$. Right: Weight recentered logarithmic barrier function (solid) and relaxed weight recentered logarithmic barrier function for $\delta \in \{0.01, 0.1, 0.5, 1\}$ for the constraint $-1\leq z \leq 2$ with $z\in\mathbb{R}$.}
  \label{fig:recMethods}
  \vspace*{-0.3cm}
 \end{figure}
\subsection{Closed-Loop Stability: Terminal Set Based Approaches} \label{sec:relBarrierStab1}
In this section we present our main results on the closed-loop stability properties of relaxed logarithmic barrier function based MPC schemes that make use of a suitable terminal set,~cf.~\cite{feller14}. 
In accordance with \cite{wills04,feller13,feller14b,feller14c}, we assume in the following that the terminal set $\mathcal{X}_f$ can be represented as
\begin{equation}\label{eq:XfPhi}
\mathcal{X}_f=\{x \in \mathbb{R}^n: \varphi(x) \leq 1\}\, ,
\end{equation}
where $\varphi: \mathbb{R}^n \to \mathbb{R}_+$ is a continuously differentiable, convex, and positive definite function that satisfies $\varphi(A_Kx) \leq \varphi(x) \ \forall x\in \mathcal{X}_f$.  
When considering ellipsoidal terminal sets as discussed in \cite{wills04} and\cite{feller13}, $\varphi(\cdot)$ may for example be chosen as $\varphi(x)=x^\TRANSP P_fx$ with a suitable $P_f \in \mathbb{S}^n_{+\!\!+}$. In \cite{feller14b} and \cite{feller14c}, it is moreover shown how smooth approximations of the Minkowski functional can be used to define $\varphi(\cdot)$ for polytopic terminal sets of the form $\mathcal{X}_f=\{x\in\mathbb{R}^n:H_fx\leq \mathds{1}\}$. \\[0.25cm]
\textit{\indent B1) Stabilization with strict constraint satisfaction}\\[0.15cm]
In the following, we will show that asymptotic stability of the closed-loop system as well as strict satisfaction of all input and state constraints can always be guaranteed for any feasible initial condition by making the relaxation parameter $\delta \in \mathbb{R}_{\pplus}$ arbitrarily small. 
The key assumption underlying the following results is that both the state and input constraint barrier functions and the terminal set constraint barrier function are relaxed.
In particular, the barrier functions for the input and state constraints are chosen according to Assumption~\ref{ass:relBarriersXU}, while the terminal cost function is given by 
\begin{equation} \label{eq:termCost_relXf}
\hat{F}(x)=x^\TRANSP\! Px+\varepsilon \hat{B}_{\mathrm{f}}(x)\, , 
\end{equation}
where $\hat{B}_{\mathrm{f}}(\cdot)$ is a relaxed logarithmic barrier function for the terminal set $\mathcal{X}_f$. 
Assuming $\delta \leq 1$, the function $\hat{B}_{\mathrm{f}}(\cdot)$ can be defined as follows.
\begin{ass} \label{ass:barrierXf}
 The barrier function $\hat{B}_{\mathrm{f}}(\cdot)$ for the terminal set $\mathcal{X}_f$ is a relaxed logarithmic barrier function of the form 
  \begin{equation}
 \hat{B}_{\mathrm{f}}(x)=\begin{cases}
               -\ln(1-\varphi(x)) & 1-\varphi(x) > \delta \\
               \, \beta(1-\varphi(x);\delta) & 1-\varphi(x) \leq \delta \, ,
              \end{cases} 
\end{equation} 
where $\delta \in \mathbb{R}_{+\!\!+}$, $\delta \leq 1$, is the relaxation parameter and $\varphi(\cdot)$ is the function defining the terminal set according to~(\ref{eq:XfPhi}).
\end{ass}
Note that the function $\hat{B}_{\mathrm{f}}(\cdot)$ is continuously differentiable, convex, and positive definite by design. 
Consider now the following definition, which introduces a lower bound for the value of the relaxed barrier functions $\hat{B}_{\mathrm{x}}(\cdot)$, $\hat{B}_{\mathrm{u}}(\cdot)$, and $\hat{B}_{\mathrm{f}}(\cdot)$ evaluated at the borders of the respective constraint sets.
\begin{defn}\label{def:betaBar}
Let the scalar $\bar{\beta}(\delta) \in \mathbb{R}_{\pplus}$ be defined as 
\begin{equation}
\bar{\beta}(\delta)=\min\left\lbrace \bar{\beta}_{\mathrm{x}}(\delta), \bar{\beta}_{\mathrm{u}}(\delta), \bar{\beta}_{\mathrm{f}}(\delta)\right\rbrace\, ,
\end{equation}
where $\bar{\beta}_{\mathrm{f}}(\delta)=\beta(0;\delta)$ and
\begin{subequations} \label{eq:betaXUEq}
\begin{align}
\bar{\beta}_{\mathrm{x}}(\delta)&=\min_{i,x}\{\hat{B}_{\mathrm{x}}(x)\vert \ {C_{\mathrm{x}}^ix=d_{\mathrm{x}}^i}\}, \quad i=1,\dots,q_{\mathrm{x}} \, , \\
\bar{\beta}_{\mathrm{u}}(\delta)&=\min_{u,j}\{\hat{B}_{\mathrm{u}}(u) \vert \ {C_{\mathrm{u}}^ju=d_{\mathrm{u}}^j}\}, \quad j=1,\dots,q_{\mathrm{u}} \, .
\end{align}
\end{subequations}
\vspace*{-0.45cm}
\end{defn} 
Note that the values $\bar{\beta}_{\mathrm{x}}(\delta)$, $\bar{\beta}_{\mathrm{u}}(\delta)$, and hence also $\bar{\beta}(\delta)$, can be computed easily for a given $\delta\in\mathbb{R}_{\pplus}$ as the optimization problems in~(\ref{eq:betaXUEq}) are convex.
Moreover, for the case of gradient recentered relaxed logarithmic barrier functions, an explicit expression for a lower bound on $\bar{\beta}(\delta)$ has been given in~\cite{feller14}.
The following Lemma shows that the sublevel sets of $\hat{B}_{\mathrm{x}}(\cdot)$, $\hat{B}_{\mathrm{u}}(\cdot)$, and $\hat{B}_{\mathrm{f}}(\cdot)$ related to $\bar{\beta}_{\mathrm{x}}(\delta)$, $\bar{\beta}_{\mathrm{u}}(\delta)$, and $\bar{\beta}_{\mathrm{f}}(\delta)$ are always contained within the sets $\mathcal{X}$, $\mathcal{U}$, and $\mathcal{X}_f$, respectively. 
\begin{lem} \label{lem:relBarrierLevelSets}
Let the Assumptions~\ref{ass:relBarriersXU} and \ref{ass:barrierXf} hold and let the values $\bar{\beta}_{\mathrm{x}}(\delta)$, $\bar{\beta}_{\mathrm{u}}(\delta)$, and $\bar{\beta}_{\mathrm{f}}(\delta)$ be defined according to Definition~\ref{def:betaBar}. Then it holds that $\mathcal{S}_{\hat{B}_{\mathrm{x}}}=\{x \in \mathbb{R}^n| \hat{B}_{\mathrm{x}}(x)\leq \bar{\beta}_{\mathrm{x}}(\delta)\}\subseteq \mathcal{X}$, $\mathcal{S}_{\hat{B}_{\mathrm{u}}}=\{u \in \mathbb{R}^m| \hat{B}_{\mathrm{u}}(u)\leq \bar{\beta}_{\mathrm{u}}(\delta)\}\subseteq \mathcal{U}$, as well as $\mathcal{S}_{\hat{B}_{\mathrm{f}}}=\{x \in \mathbb{R}^n| \hat{B}_{\mathrm{f}}(x)\leq \bar{\beta}_{\mathrm{f}}(\delta)\}\subseteq \mathcal{X}_f$.
\end{lem}
\begin{proof}
The result for the terminal set $\mathcal{X}_f$ follows directly from the fact that $\hat{B}_{\mathrm{f}}(\cdot)$ is convex and strictly monotone in the argument $\, z=1-\varphi(x)$. Considering the state constraints, let us assume that there exists an $\bar{x}\in \mathcal{S}_{B_x}$ with $\bar{x} \notin \mathcal{X}$, i.e., $C_{\mathrm{x}} \bar{x}>d_{\mathrm{x}}$. However, then there would exists a $\lambda \in (0,1)$ such that $C_{\mathrm{x}} \lambda \bar{x} =\lambda C_{\mathrm{x}} \bar{x}\leq d_{\mathrm{x}}$ and $C_{\mathrm{x}}^i \lambda \bar{x} =\lambda C_{\mathrm{x}}^i \bar{x}=d_{\mathrm{x}}^i$ for some $i=1,\dots,q_{\mathrm{x}}$. Now, due to the convexity and positive definiteness of $\hat{B}_{\mathrm{x}}(\cdot)$, it holds that $\hat{B}_{\mathrm{x}}(\lambda\bar{x})\leq \lambda \hat{B}_{\mathrm{x}}(\bar{x})<\bar{\beta}_{\mathrm{x}}(\delta)$, which obviously is a contradiction to the definition of $\bar{\beta}_{\mathrm{x}}(\delta)$. For the input constraints we can use similar arguments.
\end{proof} 
Based on the above results, we now define a set of initial conditions for which we can guarantee asymptotic stability of the origin as well as strict satisfaction of all input and state constraints. 
\begin{defn} \label{def:Xdelta}
For a given $\delta \in \mathbb{R}_{+\!\!+}$ and a corresponding $\bar{\beta}(\delta)\in \mathbb{R}_{\pplus}$ according to Definition~\ref{def:betaBar}, let the set $\hat{\mathcal{X}}_{N}(\delta)$ be defined as $\hat{\mathcal{X}}_{N}(\delta):=\{x \in \mathbb{R}^n\, |\, \hat{J}^*_N(x;\delta)\leq \varepsilon \bar{\beta}(\delta)\}$. 
\end{defn}
\begin{thm} \label{thm:exactRel}
Let Assumptions~\ref{ass:barrierMPCStab},~\ref{ass:relBarriersXU}, and \ref{ass:barrierXf} hold true. Consider problem~(\ref{eq:OptProblemRelBarrier}) with the terminal cost $\hat{F}(\cdot)$ from~(\ref{eq:termCost_relXf}) and let the set $\hat{\mathcal{X}}_N(\delta)$ be defined according to Definition~\ref{def:Xdelta}. Then, the feedback $u(k)=\hat{u}_0^*(x(k))$ asymptotically stabilizes the origin of system~(\ref{eq:discreteSystem}) under strict satisfaction of all input and state constraints for any $x(0) \in \hat{\mathcal{X}}_N(\delta)$.
\end{thm}
\begin{proof}
The proof consists of three parts and is closely related to that of Theorem~12~in~\cite{feller14}. First, we show that the underlying input, state, and terminal set constraints are not violated for any $x_0\in \hat{\mathcal{X}}_N(\delta)$; then we use standard MPC arguments to show that the value function $\hat{J}^*(x(k);\delta)$ will decrease when applying the feedback $u(k)$; finally, we use this result to conclude that the resulting input and state sequences will also be strictly feasible at all later time steps and that the origin of the closed-loop system is asymptotically stable. \\
i) Let $\hat{\boldsymbol{u}}^*(x_0)=\{\hat{u}_0^*, \dots, \hat{u}_{N-1}^*\}$, $\boldsymbol{x}^*(x_0)=\{x_0,\dots,x_N^*\}$ denote the optimal open-loop input and state sequences for a given $x_0\in \hat{\mathcal{X}}_N(\delta)$. Since the cost function in~(\ref{eq:OptProblemRelBarrier}) is a sum of positive definite terms, it holds that  $\varepsilon \hat{B}_{\mathrm{x}}(x^*_{k}) < \hat{J}_N^*(x_0;\delta)$, $\varepsilon \hat{B}_{\mathrm{u}}(\hat{u}_k^*) < \hat{J}_N^*(x_0;\delta)$, as well as $\varepsilon \hat{B}_{\mathrm{f}}(x^*_N) < \hat{J}_N^*(x_0;\delta)$ for all $x_0\in \hat{\mathcal{X}}_N(\delta)\setminus\{0\}$. Hence, $\hat{B}_{\mathrm{x}}(x^*_k) < \bar{\beta}(\delta)$, $\hat{B}_{\mathrm{u}}(\hat{u}_k^*) < \bar{\beta}(\delta)$, as well as $\hat{B}_{\mathrm{f}}(x^*_N) < \bar{\beta}(\delta)$.
 Due to Lemma~\ref{lem:relBarrierLevelSets} and the definition of~$\bar{\beta}(\delta)$, this implies that $x_k^* \in\mathcal{X}^\circ$ and $\hat{u}_k^*\in\mathcal{U}^\circ$ as well as $x_N^*\in\mathcal{X}_f^\circ$ for any $x_0 \in \hat{\mathcal{X}}_N(\delta)\setminus\{0\}$. Of course, the case $x_0=0$ is trivial.
Hence, the predicted input and state sequences are strictly feasible and the applied input results in a successor state $x_0^+=Ax_0+B\hat{u}_0^* \in\mathcal{X}_N^\circ$. \\
ii) Let us now consider $x_0^+ \in\mathcal{X}_N^\circ$. 
We can use basically the same arguments as in the proof of Theorem~\ref{thm:barrierMPCStab} to show that
\begin{equation}    \label{eq:jRelDrecease}
\hat{J}_N^*(x_0^+;\delta)-\hat{J}_N^*(x_0;\delta)\!\leq\!-\hat{\ell}(x_0,\hat{u}_0^*(x_0)) \ \forall x_0 \in \hat{\mathcal{X}}_N(\delta) .
\end{equation} 
In particular, we know that 
 $\tilde{\boldsymbol{u}}^+(x_0)=\{\hat{u}_1^*, \dots, \hat{u}_{N-1}^*, Kx_N^*\}$ and ${\boldsymbol{x}}^+(x_0)=\{x_0^+,x_2^*,\dots,x_N^*,A_Kx_N^*\}$ with $A_Kx_N^* \in \mathcal{X}_f$ are suboptimal but feasible input and state sequences for the initial state $x_0^+\in \mathcal{X}_N$.
Moreover,
\begin{subequations}
\begin{align}
&\hat{F}(A_Kx_N^{*}) -\hat{F}(x_N^*) + \hat{\ell}(x_N^*,Kx_N^*)  \\
\ &= \lVert A_K x_N^* \rVert^2_{P} - \lVert x_N^*\rVert^2_P  + \lVert x_N^* \rVert^2_Q + \lVert Kx_N^* \rVert^2_{R}  \\
\ & \quad + \varepsilon \hat{B}_K(x_N^*) + \varepsilon \hat{B}_{\mathrm{f}}(A_Kx_N^*)- \varepsilon \hat{B}_{\mathrm{f}}(x_N^*) \\
\ & \leq \varepsilon \big(\hat{B}_{\mathrm{f}}(A_Kx_N^*)- \hat{B}_{\mathrm{f}}(x_N^*)\big) \leq 0 \ \ \forall \ x_N^* \in \mathcal{X}_f \ .
\end{align}
\end{subequations}
Here, the first inequality follows from the choice of $\mathcal{X}_f$ and~$P$ according to Assumption~\ref{ass:barrierMPCStab} and the fact that $\hat{B}_K(x_N^*)\leq B_K(x_N^*)\leq x_N^{*\mathrm{T}}Mx_N^* \ \forall\, x_N^* \in \mathcal{X}_f$. The second inequality follows from the monotonicity of the relaxed logarithmic barrier function $\hat{B}_{\mathrm{f}}(\cdot)$ and the assumption that $\varphi(A_Kx_N^*)\leq\varphi(x_N^*)$ $\forall x_N^*\in \mathcal{X}_f$. Based on arguments as in the proof of Theorem~\ref{thm:barrierMPCStab}, this result can be used to prove the decrease property~(\ref{eq:jRelDrecease}). \\
iii)  The fact that the value function decreases shows that $\hat{J}_N^*(x_0^+;\delta)\leq \hat{J}_N^*(x_0;\delta) \leq \varepsilon \bar{\beta}(\delta)$ and hence $x_0^+ \in \hat{\mathcal{X}}_N(\delta)$ for any $x_0 \in \hat{\mathcal{X}}_N(\delta)$. By repeating this argument, the resulting closed-loop system state satisfies $x(k) \in \hat{\mathcal{X}}_N(\delta) \ \forall k\geq 0$ for any $x(0) \in \hat{\mathcal{X}}_N(\delta)$, which shows that all future states and inputs will be strictly feasible. 
Moreover, due to the design of the relaxed barrier functions, $\hat{J}_N^*(x;\delta)$ is a well-defined, positive definite, and radially unbounded function. 
Hence, in combination with (\ref{eq:jRelDrecease}) it can be used as a Lyapunov function, proving asymptotic stability of the origin with a guaranteed region of attraction of at least $\hat{\mathcal{X}}_N(\delta)$.
\end{proof} 
The following results state some useful properties of the region of attraction $\hat{\mathcal{X}}_N(\delta)$ and show that the feasible set $\mathcal{X}_N$ of the corresponding non-relaxed formulation can be recovered by making the relaxation parameter arbitrarily small. \vspace*{-0.15cm}
\begin{lem}[cf. \cite{feller14}]\label{lem:XbetaProp}
Let the assumptions in Theorem~\ref{thm:exactRel} hold and let the set $\hat{\mathcal{X}}_N(\delta)$ be defined according to Definition~\ref{def:Xdelta}. Then, $\hat{\mathcal{X}}_N(\delta)$ is a nonempty compact and convex set. Furthermore, $\hat{\mathcal{X}}_N(\delta) \subseteq \mathcal{X}_N^\circ$ and $\hat{\mathcal{X}}_N(\delta) \to \mathcal{X}_N$ as $\delta \to 0$. \vspace*{-0.05cm}
\end{lem}
\begin{proof}
 For any $\delta\in\mathbb{R}_{\pplus}$ satisfying Assumptions~\ref{ass:relBarriersXU}~and~\ref{ass:barrierXf}, $J_N^*(\cdot;\delta)$ is a positive definite function that is convex and radially unbounded. As $\bar{\beta}(\delta)$ is strictly positive~(see Definition~\ref{def:betaBar}), and $\varepsilon\in\mathbb{R}_{\pplus}$, the first part follows immediately.\\
 As shown in the proof of Theorem~\ref{thm:exactRel}, any initial condition $x_0\!\in\!\hat{\mathcal{X}}_N(\delta)$ results in strictly feasible input and state sequences, which implies that $\hat{\mathcal{X}}_N(\delta) \subseteq \mathcal{X}_N^\circ$. 
We now show that $\hat{\mathcal{X}}_N(\delta)$ will contain any compact subset of $\mathcal{X}_N^\circ$ if we make the relaxation parameter arbitrarily small. 
Assume $x_0 \in \mathcal{X}_N^\circ$ and let $\boldsymbol{\tilde{u}}^*(x_0)$,~$\boldsymbol{\tilde{x}}^*(x_0)$,~$\tilde{J}_N^*(x_0)$ denote the solution of the corresponding nonrelaxed problem formulation~(\ref{eq:OptProblemBarrier}). 
Then, there always exists $\delta_0(x_0):=\min\{-C_{\mathrm{x}}^i \tilde{x}_k(x_0)+d_{\mathrm{x}}^i$, $-C_{\mathrm{u}}^j \tilde{u}_k(x_0)+d_{\mathrm{u}}^j$, $1-\varphi(\tilde{x}_N^*(x_0)$, \ $i\!=\!1,\dots,q_{\mathrm{x}}$, $j\!=\!1,\dots,q_{\mathrm{u}}$, $k\!=\!0,\dots,N\!-\!1\} \, \in\mathbb{R}_{\pplus}$ with the property that the solutions of relaxed and nonrelaxed formulation will be equivalent for all $\delta\leq\delta_0(x_0)$, i.e., $\boldsymbol{\hat{u}}^*(x_0)=\boldsymbol{\tilde{u}}^*(x_0)$, $\boldsymbol{\hat{x}}^*(x_0)=\boldsymbol{\tilde{x}}^*(x_0)$, $\hat{J}_N^*(x_0;\delta)=\tilde{J}_N^*(x_0)$. Note that $\delta_0(x_0)$ is a continuous function of $x_0$ since both $\boldsymbol{\tilde{u}}^*(x_0)$ and $\boldsymbol{\hat{x}}^*(x_0)$ are continuous due to the smooth problem formulation~(\ref{eq:OptProblemRelBarrier}).
Let us further define $\delta'_0(x_0):=\max\{\delta \in\mathbb{R}_{\pplus}\,|\,\delta\leq\delta_0(x_0),\,\tilde{J}_N^*(x_0)\leq \varepsilon \bar{\beta}(\delta)\}$ with the property that $x_0\in\hat{\mathcal{X}}_N(\delta)$ $\forall\delta\leq\delta'_0(x_0)$.
As $\hat{J}_N^*(x_0;\delta)=\hat{J}_N^*(x_0;\delta_0(x_0))=\tilde{J}_N^*(x_0)$ for all $\delta\leq \delta_0(x_0)$ and $\bar{\beta}(\delta)$ is, on the other hand, continuous and strictly increasing for decreasing~$\delta$, we can state that $\delta'_0(x_0) \in\mathbb{R}_{\pplus}$ exists for any $x_0\in\mathcal{X}_N^\circ$.
Moreover, since both $\delta_0(x_0)$ and $\tilde{J}_N^*(x_0)$ are continuous, ${\delta}'_0(x_0)$ also is a continuous function of $x_0$. 
Consider now an arbitrary compact set $\mathcal{X}_0 \subseteq \mathcal{X}_N^\circ$ and define $\bar{\delta}_0\in \mathbb{R}_{\pplus}$ as $\bar{\delta}_0=\min_{x_0\in\mathcal{X}_0}{\delta}'_0(x_0)$. Due to the continuity of
 ${\delta}'_0(x_0)$ and the compactness of $\mathcal{X}_0$, this value always exists and it holds that
$\mathcal{X}_0\subseteq\hat{\mathcal{X}}_N(\delta)$ for all $\delta \leq \bar{\delta}_0$.
This implies that $\hat{\mathcal{X}}_N(\delta) \to \mathcal{X}_N$ as $\delta \to 0$, which completes the proof.
\end{proof} \vspace*{-0.1cm}
\begin{cor}\label{cor:exactRelX0}
 Consider the relaxed barrier function based MPC problem formulation~(\ref{eq:OptProblemRelBarrier}) with $\hat{F}(\cdot)$ given by~(\ref{eq:termCost_relXf}). For any compact set $\mathcal{X}_0 \subseteq \mathcal{X}_N^\circ$ there exists a $\bar{\delta}_0 \in \mathbb{R}_{+\!\!+}$ such that for any $\delta\leq\bar{\delta}_0$ and any $x(0) \in \mathcal{X}_0$ the feedback $u(k)=\hat{u}_0^*(x(k))$ asymptotically stabilizes the origin of system~(\ref{eq:discreteSystem}) under strict satisfaction of all input and state constraints.
\end{cor}
It has to be noted that the above conditions for closed-loop stability and constraint satisfaction are of course only sufficient and may be rather conservative.
In particular, for a given $\delta\!\in\!\mathbb{R}_{+\!\!+}$ the actual region of attraction of the closed-loop system may be considerably larger than the set~$\hat{\mathcal{X}}_N(\delta)$. Likewise, a very small $\delta$ may be needed to achieve $\mathcal{X}_0 \subset \hat{\mathcal{X}}_N(\delta)$ when $\mathcal{X}_0$ approaches $\mathcal{X}_N^\circ$. However, despite possible practical limitations, the presented results provide interesting insights and a theoretical justification for the use of relaxed barrier functions in the context of MPC. \\[0.25cm]
\textit{\indent B2) Global stabilization with a nonrelaxed terminal set}\\[0.15cm]
Assuming controllability of system~(\ref{eq:discreteSystem}), we will in the following present a second approach that allows to guarantee asymptotic stability for any initial condition by relaxing the barrier functions of state and input constraints while strictly enforcing the terminal set constraint.
\begin{ass} \label{ass:globalStab}
 The pair $(A,B)$ is controllable and the prediction horizon satisfies the condition $N\geq n$. 
Moreover, the terminal cost function is given as
\begin{equation}\label{eq:termCost_nonrelXf}
\hat{F}(x)=x^\TRANSP\! Px + \varepsilon {B}_f(x) \, ,
\end{equation}
where $P$ and ${B}_f(\cdot)$ are chosen according to Assumption~\ref{ass:barrierMPCStab}.
\end{ass}
\begin{thm} \label{thm:globalStabilization}
Let Assumptions~\ref{ass:barrierMPCStab},~\ref{ass:relBarriersXU}, and \ref{ass:globalStab} hold true and consider problem~(\ref{eq:OptProblemRelBarrier}) with the terminal cost $\hat{F}(\cdot)$ from~(\ref{eq:termCost_nonrelXf}). Then, independently of the relaxation parameter $\delta \in \mathbb{R}_{\pplus}$, the feedback $u(k)=\hat{u}_0^*(x(k))$ asymptotically stabilizes the origin of system~(\ref{eq:discreteSystem}) for any initial condition $x(0) \in \mathbb{R}^n$.  
\end{thm}
\begin{proof}
Due to the above controllability assumption, there exists for any $x_0 \in \mathbb{R}^n$ an input sequence $\boldsymbol{\hat{u}}(x_0)=\{\hat{u}_0, \dots, \hat{u}_{N-1}\}$ such that $x_N(\boldsymbol{\hat{u}}(x_0),x_0) \in \mathcal{X}_f^{\circ}$. Hence, $\boldsymbol{\hat{u}}^*(x_0)$ and $\hat{J}_N^*(x_0;\delta)$ are defined for any $x_0 \in \mathbb{R}^n$. 
Since the corresponding terminal state satisfies $x_N^*(x_0)\in\mathcal{X}_f^{\circ}$, the local controller $u=Kx$ can again be used to construct a feasible control sequence $\boldsymbol{\hat{u}}^+(x_0)=\{\hat{u}_1^*, \dots, \hat{u}_{N-1}^*, Kx_N^*\}$  for the successor state $x_0^+=Ax_0+B\hat{u}_0^*$. Since all parameters are chosen according to Assumption~\ref{ass:barrierMPCStab}, the same arguments as in part ii) of the proof of Theorem~\ref{thm:exactRel} can be used in order to  show that
\begin{equation}
 \hat{J}_N^*(x_0^+;\delta)- \hat{J}_N^*(x_0;\delta) \leq -\hat{\ell}_0(x_0,\hat{u}_0^*(x_0)) \ \forall x_0 \in \mathbb{R}^n \, .
\end{equation} 
Moreover, due to the design of all involved barrier functions, $\hat{J}_N^*(x;\delta)$ is a well-defined, positive definite, and radially unbounded function. 
Hence, $ \hat{J}_N^*(\cdot;\delta)$ can be employed as a Lyapunov function for proving global asymptotic stability of the origin, see part iii) of the proof of Theorem~\ref{thm:exactRel}. 
\end{proof}
\vspace*{0.2cm}
The above result illustrates how we can achieve stabilization of the origin for any initial condition, which makes the corresponding MPC scheme very robust against uncertainties, disturbances, or even infeasible start configurations. As we need to strictly enforce the terminal set constraint, the benefit of global stabilization comes with the cost of possible violations of the state and input constraints. However, using the arguments of the previous section, we can for any $\mathcal{X}_0 \subseteq \mathcal{X}_N^\circ$ still recover strict satisfaction of all state and input constraints by making the relaxation parameter $\delta$ arbitrarily small. \vspace*{-0.1cm}
\subsection{Closed-Loop Stability: Terminal Set Free Approaches} \label{sec:relBarrierStab2} 
In the previous section, we have seen how asymptotic stability and even strict constraint satisfaction can be guaranteed in the context of relaxed logarithmic barrier function based MPC by making use of suitable terminal set formulations. 
As in conventional MPC schemes, the terminal set is on the one hand used for ensuring the existence of a feasible local control law and, thus, the recursive feasibility of the corresponding open-loop optimal control problem.
On the other hand, only restricting the terminal state to a compact set around the origin allowed us to derive a quadratic upper bound for the barrier function $B_K(\cdot)$, see Assumption~\ref{ass:barrierMPCStab} and the proof of Theorem~\ref{ass:barrierMPCStab}.
In the following, we will show that the use of relaxed logarithmic barrier functions in fact allows us to circumvent these two problems and to design novel MPC approaches which not only eliminate the need for an explicit terminal set constraint but also to prove global asymptotic stability of the origin. \\[0.25cm]
\textit{\indent C1) Tail sequence based terminal cost function} \\[0.15cm]
It is a well-known result that closed-loop stability of both linear and nonlinear MPC schemes may be ensured by choosing the terminal cost as a suitable CLF that is an upper bound for the infinite-horizon cost-to-go, see~e.g.~\cite{jadbabaie01}. In the presence of input and state constraints, deriving such a function in global form is generally not possible, which directly motivates the use of a local CLF in combination with a corresponding terminal set constraint. However, when considering the relaxed problem formulation~(\ref{eq:OptProblemRelBarrier}), any input sequence which steers the state to the origin in a finite number of steps can be used to derive an upper bound on the infinite-horizon cost-to-go. 
\begin{ass} \label{ass:deadBeatStab2}
Let $\boldsymbol{{v}}(x):=\{{v}_0(x),\dots,{v}_{T-1}(x)\}$ be an input sequence which steers the state of system~(\ref{eq:discreteSystem}) to the origin in a finite number of $T\geq n$ steps for any $x \in \mathbb{R}^n$ and assume that $v_l(x)=0 \ \forall \, l=0,\dots,T-1 \Leftrightarrow x=0$. Furthermore, let $\boldsymbol{{z}}(x):=\{z_0(x),\dots,z_{T-1}(x)\}$ with $z_{{0}}(x)=x$, $z_{{l+1}}(x)=Az_{l}(x)+Bv_{l}(x)$, $l=0,\dots,T-1$, and $z_{T}(x)=0$ be the corresponding state sequence. 
\end{ass}
Based on Assumption~\ref{ass:deadBeatStab2} we propose to choose the terminal cost as
\begin{equation} \label{eq:deadBeatF}
\hat{F}(x)= \sum_{l=0}^{{T}-1} \hat{\ell}(z_l(x),v_l(x)) \, ,
\end{equation}
where $\hat{\ell}:\mathbb{R}^n\times\mathbb{R}^m \to \mathbb{R}_+$ refers to the already introduced modified stage cost based on relaxed logarithmic barrier functions for the state and input constraints,~cf.~(\ref{eq:OptProblemRelBarrier}). 
Due to the relaxation, $\boldsymbol{v}(x)$ is a valid but possibly suboptimal input sequence that steers the state from $x(0)=x$ to the origin in a finite number of steps. Consequently, $\hat{F}(\cdot)$ is an upper bound for the infinite-horizon cost-to-go, which allows us to state the following stability result.
\begin{thm} \label{thm:deadBeatStab}
 Let the Assumptions~\ref{ass:relBarriersXU} and \ref{ass:deadBeatStab2} hold and consider problem~(\ref{eq:OptProblemRelBarrier}) with the terminal cost $\hat{F}(\cdot)$ from~(\ref{eq:deadBeatF}). Then, independently of the relaxation parameter $\delta \in \mathbb{R}_{\pplus}$, 
 the feedback $u(k)=\hat{u}_0^*(x(k))$ asymptotically stabilizes the origin of system~(\ref{eq:discreteSystem}) for any initial condition $x(0) \in \mathbb{R}^n$.  
\end{thm}
\begin{proof}
Due to Assumption~\ref{ass:deadBeatStab2}, there exist for any $x_0 \in \mathbb{R}^n$ and any input sequence  $\boldsymbol{\hat{u}}(x_0)$ with resulting terminal state $x_N=x_N(\boldsymbol{\hat{u}}(x_0),x_0)$ suitable tail sequences $\boldsymbol{{v}}(x_N)$ and $\boldsymbol{{z}}(x_N)$. Moreover, as all barrier functions for the state and input constraints are relaxed, see Assumption~\ref{ass:relBarriersXU}, both the stage and terminal cost are always well defined. Hence, also $\boldsymbol{\hat{u}}^*(x_0)$ and $\hat{J}_N^*(x_0;\delta)$ are defined for any $x_0 \in \mathbb{R}^n$, which shows that problem~(\ref{eq:OptProblemRelBarrier}) always admits a feasible solution.  \\
We now want to show that the value function~$\hat{J}_N^*(x(k);\delta)$ decreases under the applied feedback for all $x(0)\in\mathbb{R}^n$.
Let $\boldsymbol{\hat{u}}^*(x_0)=\{\hat{u}_0^*, \dots, \hat{u}_{N-1}^*\}$ and $\boldsymbol{x}^*(x_0)=\{x_0,x_1^*,\dots,x_N^*\}$ denote the optimal open-loop input and state sequences for a given initial condition $x_0\in \mathbb{R}^n$ and let $\boldsymbol{v}^*(x_0):=\boldsymbol{v}(x_N^*(x_0))$ and $\boldsymbol{z}^*(x_0):=\boldsymbol{z}(x_N^*(x_0))$ be the corresponding tail sequences for the resulting predicted terminal state.
Consider now the successor state $x_0^+=Ax_0+B\hat{u}_0^*$. Clearly, $\boldsymbol{\hat{u}}^+(x_0)=\{\hat{u}_1^*,\dots,\hat{u}_{N-1}^*,v_0^*\}$ and $\boldsymbol{{v}}^+(x_0)=\{v_1^*,\dots,v_{T-1}^*,0\}$ are suboptimal input sequences which steer the state from $x_0^+$ to the origin in a finite number of $N+T-1$ steps. The resulting state sequences are given by $\boldsymbol{{x}}^+(x_0)=\{x_1^*,\dots,x_N^*,z_1^*\}$ and $\boldsymbol{{z}}^+(x_0)=\{z_1^*,\dots,z_{T-1}^*,0\}$. Using the above suboptimal input and state sequences and the fact that the tail sequences within the terminal cost are simply appended by zero values, it is straightforward to show that
$ \hat{J}_N(\boldsymbol{\hat{u}}^+(x_0),x_0^+;\delta)-\hat{J}_N^*(x_0;\delta) =-\hat{\ell}(x_0,\hat{u}_0^*(x_0)) \ \forall x_0 \in \mathbb{R}^n$.
From the suboptimality of $\boldsymbol{\hat{u}}^+(x_0)$ it follows immediately that
\begin{equation}
 \hat{J}_N^*(x_0^+;\delta)-\hat{J}_N^*(x_0;\delta)\leq -\hat{\ell}(x_0,\hat{u}_0^*(x_0)) \ \forall x_0 \in \mathbb{R}^n \ .
\end{equation} 
By the design of the relaxed barriers and the assumption that $v_l(x)=0 \ \forall \, l=0,\dots,T-1 \Leftrightarrow x=0$, the function $\hat{J}_N^*(\cdot;\delta)$ is well-defined, positive definite, and radially unbounded. Thus, it can be used as a Lyapunov function for the closed-loop system, proving global asymptotic stability of the origin.
\end{proof}
 Note that in order to ensure convexity of the terminal cost function $\hat{F}(\cdot)$, the elements of the parametrized tail input sequence $\boldsymbol{v}(x)$ should be affine in the argument $x$. In combination with the condition that $v_l(x)=0 \ \forall \, l=1,\dots,T-1 \Leftrightarrow x=0$, this in fact limits $\boldsymbol{v}(\cdot)$ to contain a sequence of linear state feedback laws, i.e. $\boldsymbol{v}(x)=\{K_0 x, \dots, K_{T-1}x\}$. In the following, we briefly discuss two different design approaches for $\boldsymbol{v}(x)$ that meet this requirement and thus allow to guarantee stability of the closed-loop system as well as convexity of the resulting overall cost function. \\[0.2cm]
The easiest way to design suitable tail sequences $\boldsymbol{v}(\cdot)$ and $\boldsymbol{z}(\cdot)$ is by making use of a linear dead-beat controller. To this end, we may choose a terminal control gain $K \in \mathbb{R}^{m \times n}$ in such a way that the matrix $A_K=A+BK$ is nilpotent, i.e., that it satisfies $A_K^r=0$ for some $r\leq n$. Under the assumption of controllability, this can for example be achieved by a suitable pole placement procedure which ensures that all eigenvalues of the matrix $A_K$ are located at the origin of the complex plane. Based on these ideas, we may set $T=n$ and choose the tail input sequence as 
 \begin{equation}\label{eq:deadBeatK}
  \boldsymbol{v}(x)=\left\lbrace {K}x,\dots,{K}\left(A+B{K}\right)^{n-1}\!x \right\rbrace \
 \end{equation} 
which results in the corresponding state sequence $\boldsymbol{z}(x)=\{x,\dots,A_K^{n-1}x\}$. Due to the design of the matrix $K$, it obviously holds that $z_n=A_K^n\!x=0$. Thus, $\boldsymbol{v}(\cdot)$ and $\boldsymbol{z}(\cdot)$ satisfy the conditions in Assumption~\ref{ass:deadBeatStab2}, and Theorem~\ref{thm:deadBeatStab} can be used to conclude stability of the closed-loop system. 
Note that, depending on the algorithm, the pole placement problem may become numerically ill-conditioned when assigning all poles to exactly the same location. Hence, for practical implementations it may be meaningful to distribute the poles of $A_K$ in an $\varepsilon$-ball around the origin. \\[0.2cm]
While the above approach allows for a rather simple design of the tail sequences $\boldsymbol{v}(\cdot)$ and $\boldsymbol{z}(\cdot)$, and thus of the terminal cost $\hat{F}(\cdot)$, the implicit requirement that the predicted terminal state is steered to the origin in at most $n$ steps might be restrictive, leading to suboptimal or even aggressive behavior of the overall closed-loop system. In the following, we present a second approach that eliminates this restriction by allowing for tail sequences with $T\geq n$ elements. 
In order to enforce in addition a certain optimality with respect to the underlying performance criterion, we propose to choose the parametrized tail input sequence $\boldsymbol{v}(\cdot)$ as the solution to the finite-horizon LQR problem with zero terminal state constraint, i.e., \\[-0.3cm]
\begin{subequations} \label{eq:zeroTstateLQR}
\begin{align}
\boldsymbol{v}(x)&=\arg\min_{\boldsymbol{v}} \sum_{l=0}^{{T}-1} \ell(z_l,v_l)  \\
\mbox{s.\,t.} \ \ & {z}_{l+1}=A{z}_{l}+B v_{l}, \ l=0,\dots,T-1\, , \\
& z_{T}=0, \ z_0=x \\[-0.4cm]
\nonumber
\end{align}
\end{subequations}
for the standard quadratic stage cost $\ell(z,v)= \| z \|_Q^2 + \|v\|_R^2 $. 
It can be shown that the solution $\boldsymbol{v}^*(x)$ to problem~(\ref{eq:zeroTstateLQR}) can be expressed as a sequence of static linear state feedbacks of the form $v_l^*(x)=K_l x$ with $l=0,\dots,T-1$. Explicit expressions for the optimal control law and the corresponding state (and costate) trajectories were given in~\cite{ntogramatzidis03} based on the solution of two unconstrained infinite-horizon LQR problems and a suitable iteration scheme.    
Furthermore, in~\cite{ntogramatzidis05} the authors present explicit solutions for more general start and end point constraints based on a parametrization of all the solutions of the extended symplectic system.
However, as discussed in the Appendix, the optimal input sequence can also be computed directly in vector form as $V^*(x)=\begin{bmatrix} v_0^{*\TRANSP}(x) & \cdots & v_{T-1}^{*\TRANSP}(x) \end{bmatrix}^\TRANSP=K_V x$ with $K_V \in \mathbb{R}^{Tm\times n}$. 
In both cases, the resulting terminal cost function can be formulated as \vspace*{-0.1cm}
\begin{equation} \label{eq:quadrFdeadBeat}
\hat{F}(x)= x^\TRANSP\! Px+ \varepsilon \sum_{l=0}^{{T}-1}\hat{B}_{\mathrm{x}}(z_l(x))+  \hat{B}_{\mathrm{u}}(v_l(x)) \, ,
\end{equation}
where $z_l(x)$ and $v_l(x)$ denote the elements of the respective state and input tail sequences and the matrix $P\in \mathbb{S}^n_{\pplus}$ can be constructed by inserting these sequences into the quadratic stage cost $\ell(\cdot,\cdot)$, see Appendix. 
As for the previous approach, global asymptotic stability of the closed-loop system can be concluded from Theorem~\ref{thm:deadBeatStab}. However, the possibility to choose the length of the tail sequences within the construction of the terminal cost typically leads to an overall improved closed-loop performance. Especially if no constraints are active or violated by the open-loop state and input tail sequences, the terminal cost function $\hat{F}(\cdot)$ based on~(\ref{eq:zeroTstateLQR}) may give a quite good approximation of the real infinite-horizon cost-to-go. 
\begin{rem}
 Note that the discussed approaches force the predicted state to the origin in $N+T$ steps and are hence similar to MPC approaches that are based on an explicit zero terminal state constraint, see e.g.~\cite{keerthi88}. 
 However, by using the proposed terminal cost function, this behavior is enforced implicitly, i.e., without introducing an explicit equality constraint within the optimization problem. Furthermore, stability can be guaranteed for a global region of attraction and the tail sequence horizon $T$ may be made arbitrary large without increasing the number of optimization variables.
\end{rem}
\textit{ C2) Quadratic terminal cost} \quad \\[0.15cm]
In the previous section, we exploited the fact that the relaxation of input and state constraint allows to apply any sequence of inputs at the end of the prediction horizon. In the following, we will show that in the presence of relaxed barrier functions we can in addition derive a \emph{global} quadratic upper bound for the combined state and input constraint barrier function, which makes it possible to use a purely quadratic terminal cost function term without the need for a corresponding terminal set constraint.  
\begin{ass}\label{ass:globalUpperBound}
Let $\hat{B}_{\mathrm{x}}(\cdot)$ and $\hat{B}_{\mathrm{u}}(\cdot)$ be relaxed gradient or weight recentered logarithmic barrier functions according to Assumption~\ref{ass:relBarriersXU}, see (\ref{eq:bxGrad}) and (\ref{eq:bxWeight}), respectively. Let the relaxing function be quadratic and given by $\beta_2(\cdot;\delta)$ from~(\ref{eq:relFunHauser}).
\end{ass}
\begin{lem} \label{lem:quadraticBound}
Let $(A,B)$ be stabilizable and let $K \in \mathbb{R}^{m\times n}$ be a corresponding stabilizing linear control gain. Furthermore, let Assumption~\ref{ass:globalUpperBound} hold and consider $\hat{B}_K(x)=\hat{B}_{\mathrm{x}}(x)+\hat{B}_{\mathrm{u}}(Kx)$ for a given $\delta\in\mathbb{R}_{\pplus}$. Then, it holds that
   \begin{equation}
   \hat{B}_K(x)\leq \ x^\TRANSP\!\left(M_{\mathrm{x}} + K^\TRANSP \!M_{\mathrm{u}}K \right) x \quad \forall x \in \mathbb{R}^n \, ,
  \end{equation} 
  where the matrices $M_{\mathrm{x}}\in \mathbb{R}^{n \times n}$ and $M_{\mathrm{u}} \in \mathbb{R}^{m\times m}$ are defined as 
$M_{\mathrm{x}}:=\frac{1}{2\delta^2}C_{\mathrm{x}}^\TRANSP\mathrm{diag}\left(\mathds{1}+w_{\mathrm{x}}\right)C_{\mathrm{x}}^{}$ and $M_{\mathrm{u}}:=\frac{1}{2\delta^2}C_{\mathrm{u}}^\TRANSP\mathrm{diag}\left(\mathds{1}+w_{\mathrm{u}}\right)C_{\mathrm{u}}^{}$, respectively. Here, $w_{\mathrm{x}} \in \mathbb{R}^{q_{\mathrm{x}}}_+$ and $w_{\mathrm{u}} \in \mathbb{R}^{q_{\mathrm{u}}}_+$ are suitable weighting vectors when considering weight recentered barrier functions, whereas $w_{\mathrm{x}}=0$, $w_{\mathrm{u}}=0$ in the gradient recentering case.
\end{lem}
\begin{proof}
 We exemplary consider the state constraints and show that $\hat{B}_{\mathrm{x}}(x)\leq \|x\|_{M_{\mathrm{x}}}^2 \ \forall x \in \mathbb{R}^n$. Based on Taylor's Theorem (see \cite[Theorem 2.1]{nocedalWright99}) we know that
 $ \hat{B}_{\mathrm{x}}(x)=\hat{B}_{\mathrm{x}}(0)+[\nabla \hat{B}_{\mathrm{x}}(0)]^\TRANSP x + \frac{1}{2} x^T \nabla^2 \hat{B}_{\mathrm{x}}(\lambda x) x $ for some $\lambda \in (0,1)$. Due the recentering of the barrier functions, the first two terms vanish and we get $\hat{B}_{\mathrm{x}}(x)= \frac{1}{2} x^T \nabla^2 \hat{B}_{\mathrm{x}}(\lambda x) x $ for some $\lambda \in (0,1)$. In particular, it holds that $\hat{B}_{\mathrm{x}}(x) \leq x^\TRANSP \! M x \ \forall x \in \mathbb{R}^n$ for any $M \in \mathbb{S}^n_{\pplus}$ satisfying $\nabla^2 B_{\mathrm{x}}(x)\preceq M \ \forall x \in \mathbb{R}^n$.
When considering the quadratic relaxing function $\beta_2(\cdot;\delta)$ from~(\ref{eq:relFunHauser}), the Hessian of $\hat{B}_{\mathrm{x}}(\cdot)$ is given by
\begin{subequations}
\begin{align}
\nabla^2 \hat{B}_{\mathrm{x}}(x)&=C_{\mathrm{x}}^\TRANSP \mathrm{diag}\big(D_1(x), \dots, D_{q_{\mathrm{x}}}(x)\big) C_{\mathrm{x}}\, , \\
 \mbox{where\ \ }D_i(x)&=\begin{cases} \frac{1+w_{\mathrm{x}}^i}{(-C_{\mathrm{x}}^ix+d_{\mathrm{x}}^i)^2} & -C_{\mathrm{x}}^ix+d_{\mathrm{x}}^i > \delta \\
 \ \ \frac{1+w_{\mathrm{x}}^i}{\delta^2} & -C_{\mathrm{x}}^ix+d_{\mathrm{x}}^i \leq \delta \, . \end{cases}
\end{align} 
\end{subequations}
\noindent Note that $w_{\mathrm{x}}=0$ in the case of gradient recentered barrier functions. Since $D_i(x) \leq \frac{1+w_{\mathrm{x}}^i}{\delta^2} \ \forall x\in\mathbb{R}^n$, it follows immediately that 
$ \nabla^2 \hat{B}_{\mathrm{x}}(x) \preceq \frac{1}{\delta^2} C_{\mathrm{x}}^\TRANSP \mathrm{diag}\left(\mathds{1}+w_{\mathrm{x}} \right) C_{\mathrm{x}} \ \forall x \in \mathbb{R}^n$.
Combining this upper bound on the Hessian with the previous arguments, it follows that $\hat{B}_{\mathrm{x}}(x) \leq  x^\TRANSP \! M_{\mathrm{x}}\,x \ \forall x \in \mathbb{R}^n$ with $M_{\mathrm{x}}=\frac{1}{2\delta^2} C_{\mathrm{x}}^\TRANSP \mathrm{diag}\left(\mathds{1}+w_{\mathrm{x}} \right) C_{\mathrm{x}}$. Similarly, it is straightforward to show that $\hat{B}_{\mathrm{u}}(Kx) \leq  x^\TRANSP \! K^\TRANSP M_{\mathrm{u}} K x \ \forall x \in \mathbb{R}^n$ with $M_{\mathrm{u}}$ as defined above.
\end{proof} \vspace*{0.3cm}
\noindent Let now the controller matrix $K$ and the matrix $P \in \mathbb{S}^n_{\pplus}$ be solutions to the modified Riccati equation
\begin{subequations} \label{eq:riccatiP}
\begin{align} 
K&=-\left(R+B^\TRANSP PB+\varepsilon M_{\mathrm{u}}\right)^{-1}B^\TRANSP PA  \\
P&=A_{K}^{\TRANSP} P A_K^{}+K^\TRANSP \!(R+\varepsilon M_{\mathrm{u}})K+Q + \varepsilon M_{\mathrm{x}} \ ,
\end{align}
\end{subequations}
where $M_{\mathrm{x}}$ and $M_{\mathrm{u}}$ are defined according to Lemma~\ref{lem:quadraticBound}.
We then propose to choose the terminal cost function as 
\begin{equation} \label{eq:quadraticF}
\hat{F}(x)=x^\TRANSP\!P\,x \, , 
\end{equation}
where $P\in \mathbb{S}^n_{\pplus}$ is the solution to~(\ref{eq:riccatiP}b).
Note that the controller gain~$K$ is in principle arbitrary. However, the above choice results in a minimal value of the terminal cost function and ensures that for $\varepsilon \to 0$ or in the absence of constraints, $K$ and $P$ will reduce to the solution of the unconstrained LQR problem.  
We can now state the following stability result.
\begin{thm} \label{thm:quadrBoundStab}
 Let Assumption~\ref{ass:globalUpperBound} hold true and consider problem~(\ref{eq:OptProblemRelBarrier}) with the terminal cost function~$\hat{F}(\cdot)$ from~(\ref{eq:quadraticF}).
 Then, independently of the relaxation parameter $\delta \in \mathbb{R}_{\pplus}$, 
 the feedback $u(k)=\hat{u}_0^*(x(k))$ asymptotically stabilizes the origin of system~(\ref{eq:discreteSystem}) for any initial condition $x(0) \in \mathbb{R}^n$.  
\end{thm}
\begin{proof}
 Again, both $\boldsymbol{\hat{u}}^*(x_0)$ and $\hat{J}_N^*(x_0;\delta)$ are defined for any $x_0 \in \mathbb{R}^n$ due to the relaxed state and input constraints. Based on the quadratic terminal cost function above, it is now straightforward to show that the value function~$\hat{J}_N^*(x(k);\delta)$ decreases under the applied MPC feedback for all $x(k)\in\mathbb{R}^n$.
For any $x_0 \in \mathbb{R}^n$ there exist optimal input and state sequences $\boldsymbol{\hat{u}}^*(x_0)=\{\hat{u}_0^*, \dots, \hat{u}_{N-1}^*\}$ and $\boldsymbol{x}^*(x_0)=\{x_0,x_1^*,\dots,x_N^*\}$ as well as the suboptimal input and state sequences $\boldsymbol{\hat{u}}^+(x_0)=\{\hat{u}_1^*,\dots,\hat{u}_{N-1}^*,Kx_N^*\}$ and $\boldsymbol{x}^+(x_0)=\{x_1^*,\dots,x_N^*,A_Kx_N^*\}$  for the successor state $x_0^+=Ax_0+B\hat{u}_0^*$ with $\hat{J}_N^*(x_0^+;\delta)-\hat{J}_N^*(x_0;\delta)\leq \hat{J}_N(\boldsymbol{\hat {u}}^+(x_0),x_0^+,\delta)-\hat{J}_N^*(x_0;\delta)$.
This gives us
\begin{align}
&\hat{J}_N^*(x_0^+;\delta)-\hat{J}_N^*(x_0;\delta) \leq \lVert A_K x_N^* \rVert^2_{P} - \lVert x_N^*\rVert^2_P \ \dots\\
\ &   \ +\lVert x_N^* \rVert^2_Q + \lVert Kx_N^* \rVert^2_{R}  + \varepsilon B_K(x_N^*)  
\ \leq  0 \ \ \forall \ x_N^* \in \mathbb{R}^n \, , \nonumber
\end{align}
where we used the global quadratic bound on $B_K(\cdot)$ from Lemma~\ref{lem:quadraticBound} and the choice of the terminal cost matrix~$P$, see~(\ref{eq:riccatiP}b). Thus, $\hat{J}_N^*(\cdot;\delta)$ can again be used as a Lyapunov function for proving global asymptotic stability.
\end{proof} 
In summary, we can conclude that the concept of relaxed logarithmic barrier functions allows us to design globally stabilizing MPC schemes without making use of an explicit terminal set or state constraint. Instead, the presented approaches are based on a suitable design of the respective terminal cost function term that heavily exploits the properties and advantages of the underlying barrier function relaxation. 
Thus, the presented results may be seen as a barrier function based counterpart to existing terminal set free MPC approaches relying on suitable terminal cost {functions~(\!\cite{jadbabaie01,limon06})}, a sufficiently large prediction horizon~(\!\cite{chmielewski96,scokaert98,bemporadExplicitLQR}), or particular controllability assumptions~(\!\cite{grimm05,gruene12,boccia14}). \\
However, with the exception of Theorem~\ref{thm:exactRel}, no guarantees on the satisfaction of state and input constraints have been discussed so far. In fact, the presented global stability results can only be achieved since the relaxed barrier functions allow for in principle arbitrarily large constraint violations. In the next section, we will show that the maximal violation of state and input constraints in closed-loop operation is always bounded, and we discuss how we can compute and control it a priori for a given set of initial conditions by adjusting the relaxation parameter $\delta\in\mathbb{R}_{\pplus}$.

\subsection{Maximal Constraint Violation Guarantees} \label{sec:relBarrierConstr}
As outlined above, one of the main advantages of model predictive control is given by its ability to deal with input and state constraints. 
On the other hand, possible violations of the corresponding constraints seem to be inherent to the concept of relaxed barrier functions.
In the following, we will show how guarantees on the strict satisfaction or the maximal violation of input and state constraints can be given for the globally stabilizing relaxed logarithmic barrier function based MPC schemes from the previous section. 
We start with the following Lemma which allows to upper bound the values of the relaxed barrier functions in closed-loop operation by an expression that depends on the initial condition.
\begin{lem}\label{lem:relBarrierBounds}
Let problem~(\ref{eq:OptProblemRelBarrier}) be formulated based on one of the globally stabilizing MPC schemes discussed above (Theorems~\ref{thm:globalStabilization},~\ref{thm:deadBeatStab},~\ref{thm:quadrBoundStab}).
Furthermore, let $\boldsymbol{x}_{cl}=\{x(0),x(1),\dots\}$ and $\boldsymbol{u}_{cl}=\{u(0),u(1),\dots \}$ with $u(k)=\hat{u}_0^*(x(k))$ for $k\geq 0$ denote the resulting state and input trajectories of the closed-loop system. Then, for any initial condition $x_0=x(0) \in \mathbb{R}^n$ and any $k\geq 0$ it holds that 
\vspace*{-0.35cm}
\begin{subequations}
\begin{align}
  \hat{B}_{\mathrm{x}}(x(k))&\leq \frac{1}{\varepsilon} \left(\hat{J}_N^*(x_0;\delta)-x_0^\TRANSP  P_{\mathrm{uc}}^{*} \, x_0 \right)\, , \\
  \hat{B}_{\mathrm{u}}(u(k))&\leq \frac{1}{\varepsilon} \left(\hat{J}_N^*(x_0;\delta)-x_0^\TRANSP  P_{\mathrm{uc}}^* \,\, x_0 \right)\, ,
\end{align}
\end{subequations}
where $P_{\mathrm{uc}}^* \in \mathbb{S}^n_{\pplus}$ is the solution to the discrete-time algebraic Riccati equation related to the infinite-horizon LQR problem.
\end{lem}
\begin{proof}
Based on the respective Theorems we know that
 $\hat{J}_N^*(x(k+1);\delta)-\hat{J}_N^*(x(k);\delta) \leq -\hat{\ell}(x(k),u(k))$ for any $k\geq0$ and any $x(0) \in \mathbb{R}^n$. In particular, this ensures that the closed-loop system is asymptotically stable and that $\lim_{k\to \infty}\hat{J}_N^*(x(k);\delta)=0$. Summing up over all future sampling instants and using a telescoping sum on the left hand side, we get that $\hat{J}_N^*(x(0);\delta) \geq \sum_{k=0}^\infty \hat{\ell}(x(k),u(k))$. Furthermore, $\sum_{k=0}^\infty \hat{\ell}(x(k),u(k))= \sum_{k=0}^\infty \ell(x(k),u(k))+\varepsilon \sum_{k=0}^\infty \hat{B}_{\mathrm{x}}(x(k))+\hat{B}_{\mathrm{u}}(u(k))$ and $ \sum_{k=0}^\infty \ell(x(k),u(k))\geq x(0)^\TRANSP \! P_{\mathrm{uc}}^*\, x(0)$ due to the optimality of the unconstrained infinite-horizon LQR solution. In combination this yields
 \begin{equation}
  \hat{J}_N^*(x(0);\delta) \geq x(0)^\TRANSP \! P_{\mathrm{uc}}^* \,x(0)+\varepsilon \sum_{k=0}^\infty \hat{B}_{\mathrm{x}}(x(k))+\hat{B}_{\mathrm{u}}(u(k))
 \end{equation} 
 and finally, since all terms in the sum on the right hand side are positive definite, 
 $\varepsilon \hat{B}_{\mathrm{x}}(x(k)) \leq   \hat{J}_N^*(x(0);\delta)- x(0)^\TRANSP \! P_{\mathrm{uc}}^* x(0)$ as well as 
 $ \varepsilon \hat{B}_{\mathrm{u}}(u(k)) \leq   \hat{J}_N^*(x(0);\delta)- x(0)^\TRANSP \! P_{\mathrm{uc}}^* x(0)$ $\forall k\geq0$.
\end{proof} \vspace*{0.1cm}
\noindent For ease of notation, let now the scalar $\hat{\alpha}(x_0;\delta) \in \mathbb{R}_{+}$ be defined as \vspace*{-0.15cm}
\begin{equation} \label{eq:alphaEq}
\hat{\alpha}(x_0;\delta):= \hat{J}_N^*(x_0;\delta)-x_0^\TRANSP P_{\mathrm{uc}}^*\, x_0 \, .
\end{equation}
As the relaxed barrier function are positive definite and radially unbounded, Lemma~\ref{lem:relBarrierBounds} implies that also the violations of the corresponding constraints are bounded. 
In particular, for any $\varepsilon\in\mathbb{R}_{\pplus}$, $\delta\in\mathbb{R}_{\pplus}$ and any initial condition $x_0 =x(0)\in \mathbb{R}^n$, upper bounds for the maximal violations of state and input constraints are given by
\begin{subequations}\label{eq:maxConstrViol}
\begin{align}
 \hat{z}_{x}^i(x_0;\delta)= &\max_{x} \bigg\lbrace C_{\mathrm{x}}^i x -d_{\mathrm{x}}^i\, \big\vert \ \hat{B}_{\mathrm{x}}(x)\leq \frac{\hat{\alpha}(x_0;\delta)}{\varepsilon} \bigg\rbrace  \\
  \hat{z}_{u}^j(x_0;\delta)= &\max_{u} \bigg\lbrace C_{\mathrm{u}}^j x -d_{\mathrm{u}}^j\, \big\vert  \ \hat{B}_{\mathrm{u}}(u)\leq \frac{\hat{\alpha}(x_0;\delta)}{\varepsilon} \bigg\rbrace 
\end{align}
\end{subequations}
for $i=1,\dots,q_{\mathrm{x}}$ and $j=1,\dots,q_{\mathrm{u}}$, where $\hat{\alpha}(x_0;\delta)$ is defined according to~(\ref{eq:alphaEq}). Based on this, we can formulate the following result on the maximal constraint violations that can occur in closed-loop operation. 
\begin{thm} \label{thm:maxConstrViolation}
Let problem~(\ref{eq:OptProblemRelBarrier}) be formulated based on one of the globally stabilizing MPC schemes discussed above (Theorems~\ref{thm:globalStabilization},~\ref{thm:deadBeatStab},~\ref{thm:quadrBoundStab}) and let $\boldsymbol{x}_{cl}=\{x(0),x(1),\dots\}$ and $\boldsymbol{u}_{cl}=\{u(0),u(1),\dots \}$ with $u(k)=\hat{u}_0^*(x(k))$ for $k\geq 0$ denote the resulting closed-loop state and input trajectories. 
Then, for any $x_0=x(0) \in \mathbb{R}^n$ and any $k\geq 0$ it holds that 
\begin{equation}
    C_{\mathrm{x}} x(k)\leq d_{\mathrm{x}} + \hat{z}_{\mathrm{x}}(x_0,\delta), \quad  C_{\mathrm{u}} u(k)\leq d_{\mathrm{u}} + \hat{z}_{\mathrm{u}}(x_0,\delta),
\end{equation} 
where the elements of the maximal constraint violation vectors $\hat{z}_{\mathrm{x}}(x_0;\delta) \in \mathbb{R}^{q_{\mathrm{x}}}$ and $\hat{z}_{\mathrm{u}}(x_0;\delta) \in \mathbb{R}^{q_{\mathrm{u}}}$ are given by~(\ref{eq:maxConstrViol}).
\end{thm}
\begin{proof}
Follows directly from Lemma~\ref{lem:relBarrierBounds} and the definition of $\hat{\alpha}(x_0;\delta)$ and $\hat{z}_{x}^i(x_0;\delta)$, $ \hat{z}_{u}^j(x_0;\delta)$ in (\ref{eq:alphaEq}) and~(\ref{eq:maxConstrViol}).
\end{proof} 
Note that the optimization problems in~(\ref{eq:maxConstrViol}) are convex as $\hat{B}_{\mathrm{x}}(\cdot)$ and $\hat{B}_{\mathrm{u}}(\cdot)$ are convex functions. 
Moreover, when considering gradient recentered logarithmic barrier functions, we can exploit the fact that $\hat{B}_{\mathrm{x}}(\cdot)$ and $\hat{B}_{\mathrm{u}}(\cdot)$ consist only of positive definite terms. In this case, upper bounds for the maximal constraint violations are given by $\hat{z}_{\mathrm{x}}^i(x_0;\delta)\leq \max \{-\bar{z}_{\mathrm{x}}^{i},0\}$, where $\bar{z}_{\mathrm{x}}^{i}$ is given as
  \begin{equation} \label{eq:ziEq}
 \bar{z}_{\mathrm{x}}^{i}\!=\min\left\lbrace z\,\big|\,\beta({z};\delta)+\ln(d_{\mathrm{x}}^i)+\frac{z}{d_{\mathrm{x}}^i}-1=\frac{\hat{\alpha}(x_0;\delta)}{\varepsilon}\right\rbrace 
  \end{equation} 
for $i=1,\dots,q_{\mathrm{x}}$ and similar for the input constraints,~cf.~(\ref{eq:bxGrad}).
If the relaxing function $\beta(\cdot\,;\delta)$ is chosen as $\beta_e(\cdot\,;\delta)$ or $\beta_k(\cdot\,;\delta)$ with $k>2$, a nonlinear equation solver can be used to find suitable solutions $\bar{z}_{\mathrm{x}}^{i}$. However, in the case of $\beta(\cdot\,;\delta)=\beta_k(\cdot\,;\delta)$ with $k=2$, the equality constraint in~(\ref{eq:ziEq}) reduces to a quadratic equation in ${z}$ and a closed form expression of the maximal constraint violation can be given as
\begin{equation}
\bar{z}_{\mathrm{x}}^i= \delta\left(\gamma_{i,1}-\sqrt{\gamma_{i,1}^2-\gamma_{i,2}} \right), \ i=1,\dots,q_{\mathrm{x}} \, ,
\end{equation} 
where $\gamma_{i,1}:=2-\frac{\delta}{d_{\mathrm{x}}^i}$ and $\gamma_{i,2}:=1+2\ln\left(\frac{d_{\mathrm{x}}^i}{\delta}\right)-\frac{2}{\varepsilon}\,\hat{\alpha}(x_0;\delta)$, see also Lemma~21 in \cite{feller14}.
Note that the values for the maximal constraint violations obtained above may also be negative. In this case, the corresponding constraints will not be violated but rather satisfied with the respective safety margin. 
An important consequence of Theorem~\ref{thm:maxConstrViolation} is that, as stated in the following Lemma, there always exists a set of initial conditions for which the input and state constraints will not be violated at all.
\begin{lem}
Let problem~(\ref{eq:OptProblemRelBarrier}) be formulated based on one of the globally stabilizing MPC schemes discussed above (Theorems~\ref{thm:globalStabilization},~\ref{thm:deadBeatStab},~\ref{thm:quadrBoundStab}) and let $\boldsymbol{x}_{cl}=\{x(0),x(1),\dots\}$ and $\boldsymbol{u}_{cl}=\{u(0),u(1),\dots \}$ with $u(k)=\hat{u}_0^*(x(k))$ for $k\geq 0$ denote the resulting closed-loop state and input trajectories. Furthermore, for $\delta \in \mathbb{R}_{\pplus}$ let the set $\hat{\mathcal{X}}'_{N}(\delta)$ be defined as 
 \begin{equation}
  \hat{\mathcal{X}}'_{N}(\delta):=\left\lbrace x \in \mathcal{X}\, |\, \hat{\alpha}(x,\delta) \leq \varepsilon \bar{\beta}'(\delta) \right\rbrace \, ,
 \end{equation} 
 with $\bar{\beta}'(\delta):=\min\{\bar{\beta}_{\mathrm{x}}(\delta), \bar{\beta}_{\mathrm{u}}(\delta)\}$, where $\bar{\beta}_{\mathrm{x}}(\delta)$ and $\bar{\beta}_{\mathrm{u}}(\delta)$ are defined according to Definition~\ref{def:betaBar}. Then, for any initial condition $x(0)\in \hat{\mathcal{X}}'_{N}(\delta)$ and any $k\geq 0$ it holds that $C_{\mathrm{x}} x(k) \leq d_{\mathrm{x}}$ as well as $C_{\mathrm{u}} u(k) \leq d_{\mathrm{u}}$.
\end{lem}
\begin{proof}
For any $x_0=x(0) \in   \hat{\mathcal{X}}'_{N}(\delta)$ and all $k\geq 0$ it holds due to Lemma~\ref{lem:relBarrierBounds} and the definition of $\bar{\beta}'(\delta)$ that $\varepsilon \hat{B}_{\mathrm{x}}(x(k)) \leq $ $\hat{\alpha}(x_0;\delta)$ $\leq \varepsilon \bar{\beta}_{\mathrm{x}}(\delta)$ and $\varepsilon \hat{B}_{\mathrm{u}}(u(k)) \leq \hat{\alpha}(x,\delta) \leq \varepsilon \bar{\beta}_{\mathrm{u}}(\delta)$. In combination with the definition of $\bar{\beta}_{\mathrm{x}}(\delta)$ and $\bar{\beta}_{\mathrm{u}}(\delta)$, it follows directly that $x(k) \in \mathcal{X}$, $u(k) \in \mathcal{U}$ $\forall k\geq 0$, cf.~Lemma~\ref{lem:relBarrierLevelSets}. 
\end{proof} 
\noindent One may now ask how large we can make the set~$\hat{\mathcal{X}}'_{N}(\delta)$ of initial conditions for which  strict satisfaction of all input and state constraints is guaranteed. In the following, we give an answer to this question for each of the different MPC approaches discussed above. In particular, we show that it is always possible to recover the feasible set of a suitable corresponding nonrelaxed MPC formulation.
\begin{thm} \label{thm:xBeta}
 Let problem~(\ref{eq:OptProblemRelBarrier}) be formulated based on one of the globally stabilizing MPC schemes discussed above (Theorems~\ref{thm:globalStabilization},~\ref{thm:deadBeatStab},~\ref{thm:quadrBoundStab}). 
Moreover, let $\mathcal{X}_N$ be defined as 
\begin{equation} \label{eq:xN_def_tSet}
 \mathcal{X}_N=\{x \in \mathcal{X}: \exists\, \boldsymbol{u} \text{\ s.t.\ } x_k \in \mathcal{X}, u_k\in \mathcal{U}, x_N \in \mathcal{X}_f\} 
\end{equation} 
when considering the approach based on a nonrelaxed terminal set constraint~(Theorem~\ref{thm:globalStabilization}), as
\begin{equation} \label{eq:xN_def_deadBeat}
\begin{aligned}
 \mathcal{X}_N=\{x \in \mathcal{X}:&\ \exists\, \boldsymbol{u} \text{\ s.t.\ } x_k \in \mathcal{X}, u_k\in \mathcal{U}, \\
 & \qquad \quad \ z_l(x_N) \in \mathcal{X}, v_l(x_N) \in \mathcal{U}\} \quad
\end{aligned} 
\end{equation}
when considering the approach based on auxiliary tail sequences $z_l(\cdot)$ and $v_l(\cdot)$ with $l=1,\dots,T-1$ (Theorem~\ref{thm:deadBeatStab}), and as 
\begin{equation} \label{eq:xN_def_quadr}
 \mathcal{X}_N=\{x \in \mathcal{X}: \exists\, \boldsymbol{u} \text{\ s.t.\ } x_k \in \mathcal{X}, u_k\in \mathcal{U}, x_N = 0\} \ \ 
\end{equation} 
when considering the approach based on a purely quadratic terminal cost function~(Theorem~\ref{thm:quadrBoundStab}), where in all three cases $x_0=x$ and $x_k=x_k(\boldsymbol{u},x)$ for $k=1,\dots,N$. Then, for any compact set of initial conditions $\mathcal{X}_0 \subseteq \mathcal{X}_N^\circ$ there exists a $\bar{\delta}_0 \in \mathbb{R}_{\pplus}$ such that $\mathcal{X}_0\subseteq   \hat{\mathcal{X}}'_{N}(\delta)$ for all $\delta \leq \bar{\delta}_0$.
\end{thm}
\begin{proof}
We first consider the approaches based on Theorem~\ref{thm:globalStabilization} and Theorem~\ref{thm:deadBeatStab}. The proof is closely related to that of Lemma~\ref{lem:XbetaProp} and only a sketch is given here. In particular, it can be shown that $\hat{\mathcal{X}}'_{N}({\delta})$ is a nonempty and compact set with $\hat{\mathcal{X}}'_{N}(\delta)\subseteq \mathcal{X}_N^\circ$ and $\hat{\mathcal{X}}'_{N}(\delta) \to \mathcal{X}_N$ as $\delta \to 0$. The result is again based an the fact that the existence of a strictly feasible solution implies that for any $x_0 \in \mathcal{X}_N^\circ$ there exists a ${\delta}_0(x_0)$ such that $\hat{\alpha}(x_0;\delta)$ will stay constant for all $\delta \leq {\delta}_0$, whereas $\bar{\beta}(\delta)$ can be made arbitrarily large as $\delta\to 0$.\\
When considering the approach based on Theorem~\ref{thm:quadrBoundStab} the problem occurs that the matrix $P$ from (\ref{eq:riccatiP}), and hence the quadratic terminal cost function $\hat{F}(x)=x^TPx$, will grow without bound for $\delta \to 0$. Also, in this case it does not necessarily hold that $\hat{\mathcal{X}}'_{N}(\bar{\delta}) \subset \mathcal{X}_N$ as the resulting optimal state and input sequences for a given $x_0 \in   \hat{\mathcal{X}}'_{N}(\delta)$ will in general not satisfy the additional constraint $x_N=0$. 
However, for any $x_0 \in \mathcal{X}_N^\circ$ there exists input and state sequences $\boldsymbol{\bar{u}}(x_0)$ and $\boldsymbol{\bar{x}}(x_0)$ which strictly satisfy the conditions specified in~(\ref{eq:xN_def_quadr}), in particular $x_N(\boldsymbol{\bar{u}}(x_0),x_0)=0$.
For these sequences, we can always find a ${\delta}_0(x_0) \in \mathbb{R}_{\pplus}$ such that $\hat{J}_N(\boldsymbol{\bar{u}}(x_0),x_0;\delta) =\tilde{J}_N(\boldsymbol{\bar{u}}(x_0),x_0)$ for all $\delta \leq \delta_0(x_0)$, where $\tilde{J}_N(\boldsymbol{\bar{u}}(x_0),x_0)$ denotes the value function of a nonrelaxed problem formulation with the additional constraint $x_N=0$, cf. the proof of Lemma~\ref{lem:XbetaProp}. Furthermore, due to the suboptimality of the input sequence~$\boldsymbol{\bar{u}}(x_0)$, it holds that $\hat{J}_N^*(x_0;\delta)\leq \hat{J}_N(\boldsymbol{\bar{u}}(x_0),x_0;\delta)$. 
 Hence, $ \hat{J}_N^*(x_0;\delta) \leq \tilde{J}_N(\boldsymbol{\bar{u}}(x_0),x_0)$ for all $\delta \leq \delta_0(x_0)$, which shows that for any $x_0 \in \mathcal{X}_N^\circ$ the value function $\hat{J}_N^*(x_0;\delta)$, and hence also the expression $\hat{\alpha}(x_0;\delta)$, will stay bounded as $\delta \to 0$. As, on the other hand, $\bar{\beta}(\delta)$ increases without bound and both $\delta_0(x_0)$ and $\hat{J}_N^*(x_0;\delta)$ are continuous functions, there exists for any compact set $\mathcal{X}_0 \subseteq \mathcal{X}_N^\circ$ a $\bar{\delta}_0 \in \mathbb{R}_{\pplus}\leq \min_{x_0\in \mathcal{X}_0}\delta_0(x_0)$ such that $\mathcal{X}_0 \subseteq \hat{\mathcal{X}}'_N(\delta) \ \forall \delta\leq\bar{\delta}_0$, see also the proof of Lemma~\ref{lem:XbetaProp}.
\end{proof} \vspace*{0.1cm}
\noindent Note that for a general stabilizable system, the set~$\mathcal{X}_N$ in~(\ref{eq:xN_def_quadr}) may be restricted to a lower-dimensional subspace of $\mathbb{R}^n$ or it may even be empty. However, in case of a controllable system it is straightforward to show that $\mathcal{X}_N$ is a nonempty polytope. \\[0.2cm]
\noindent In general, the parameter $\bar{\delta}_0$ that is sufficient for ensuring strict constraint satisfaction based on the above results may be very small. For practical applications, it might therefore be reasonable to enforce the satisfaction of input and state constraints only with a predefined tolerance. 
Based on the above arguments, we can easily state the following result.
\begin{cor}
 Let problem~(\ref{eq:OptProblemRelBarrier}) be formulated based on one of the globally stabilizing MPC schemes discussed above (Theorems~\ref{thm:globalStabilization},~\ref{thm:deadBeatStab},~\ref{thm:quadrBoundStab}) 
  and let the respective sets $\mathcal{X}_N$ be given according to Theorem~\ref{thm:xBeta}. Then, for any given constraint violation tolerance $\hat{z}_{\textrm{tol}}\in \mathbb{R}_{+}$ and any given set of initial conditions $\mathcal{X}_0 \subseteq \mathcal{X}_N^\circ$ there exists a $\bar{\delta}_0 \in \mathbb{R}_{\pplus}$ such that for all $\delta \leq \bar{\delta}_0$ and any initial condition $x(0) \in \mathcal{X}_0$ the maximal possible constraint violation is less than $\hat{z}_{\textrm{tol}}$, i.e., that for all $k\geq 0$
 \begin{equation}
   C_{\mathrm{x}} x(k) \leq d_{\mathrm{x}}+\hat{z}_{\textrm{tol}}\mathds{1}, C_{\mathrm{u}} u(k) \leq d_{\mathrm{u}}+\hat{z}_{\textrm{tol}}\mathds{1} \ .
 \end{equation} 
\end{cor}
\noindent Based on the above results, we can formulate the following iterative algorithm that allows to determine a priori a sufficiently small relaxation parameter for ensuring satisfaction of a given maximal constraint violation tolerance for a given set of initial conditions.
\newpage
\hrule height 0.3pt
\vspace*{0.1cm}
\noindent \textbf{Algorithm 1} \textit{(Compute $\bar{\delta}_0$ for $\mathcal{X}_0$ and $\hat{z}_{\textrm{tol}} \in\mathbb{R}_+$)} \\
\vspace*{-0.25cm}
\hrule height 0.1pt
\vspace*{0.1cm}
\begin{algorithmic}[1]
\REQUIRE problem formulation, set $\mathcal{X}_0$, tolerance $\hat{z}_{\textrm{tol}}$
\ENSURE  parameter $\bar{\delta}_0 \in \mathbb{R}_{+\!\!+}$ s.t. the maximal constraint violation is bounded by $\hat{z}_{\textrm{tol}}$ for any $x(0) \in \mathcal{X}_0$ \\
\STATE choose initial $\varepsilon, \delta \in \mathbb{R}_{\pplus}$ and set up problem~(\ref{eq:OptProblemRelBarrier})
\REPEAT
\STATE $\llcorner$ decrease relaxation parameter: $\delta\leftarrow\gamma\delta$, $\gamma \in (0,1)$\\
\STATE $\llcorner$ determine $\hat{\alpha}(\mathcal{X}_0,\delta)=\max_{x \in \mathcal{X}_0} \hat{\alpha}(x,\delta)$ based on~(\ref{eq:alphaEq})
\STATE $\llcorner$ compute $\hat{z}_{\mathrm{x}}^i(\mathcal{X}_0,\delta)$, $\hat{z}_{\mathrm{u}}^j(\mathcal{X}_0,\delta)$ from~(\ref{eq:maxConstrViol}) with $\hat{\alpha}(\mathcal{X}_0,\delta)$
\UNTIL{$ \hat{z}_{\mathrm{x}}^i(\mathcal{X}_0;\delta)\leq\hat{z}_{\textrm{tol}}$ and $\hat{z}_{\mathrm{u}}^j(\mathcal{X}_0;\delta)\leq\hat{z}_{\textrm{tol}}$ holds $\forall\ i, j$}
\end{algorithmic}
\hrule height 0.3pt 
 \vspace*{0.35cm}
\begin{rem}
Up to now, it has not been clarified whether the function $\hat{\alpha}(x;\delta)$ that is maximized in step~4 of Algorithm~1 is convex.
However, for all tested parameter configurations and examples with convex $\mathcal{X}_0$, the maximal value of $\hat{\alpha}(x;\delta)$ was in fact always attained at on of the vertices of the set~$\mathcal{X}_0$. Moreover, a conservative solution can always be found by evaluating a convex upper bound of $\hat{\alpha}(x;\delta)$, e.g., the value function $\hat{J}_N^*(x;\delta)$ itself, at the vertices of the set~$\mathcal{X}_0$.  
\end{rem}
Summarizing, we can state that, despite the use of the relaxed barrier functions, the maximal possible violation of input and state constraints in closed-loop operation is bounded and depends directly on the choice of the relaxation parameter~$\delta$.
Furthermore, the presented results allow to compute an estimate for the maximal constraint violation a priori or to determine a suitable relaxation parameter that guarantees satisfaction of a given constraint violation tolerance. 
Note that certain constraints can be prioritized by making use of different $\delta_i$ in the relaxation. This important fact may for example be used in order to enforce satisfaction of physically motivated hard input constraints. 
\subsection{Closed-Loop Performance and Robustness} \label{sec:relBarrierOpt} 
In this section, we briefly discuss some aspects concerning the performance and robustness properties of the presented relaxed barrier function based MPC approaches. Of course, a thorough investigation of these issues is well beyond the scope of this paper and can be considered as future work.
We begin with the following arguments, which show that the presented relaxed barrier function based MPC schemes will always recover the closed-loop performance of a related MPC scheme based on nonrelaxed barrier functions if the relaxation parameter~$\delta$ is small enough as well as that of a conventional linear MPC scheme when the barrier weighting~$\varepsilon$ goes to zero. \\[0.2cm]
 Let the relaxed barrier function based MPC problem~(\ref{eq:OptProblemRelBarrier}) be formulated based on one of the stabilizing design approaches discussed above (Theorem~\ref{thm:exactRel},~\ref{thm:globalStabilization},~\ref{thm:deadBeatStab}, or~\ref{thm:quadrBoundStab}), and let the set $\mathcal{X}_N$ be defined according to~(\ref{eq:xN_def_tSet}) for the terminal set based approaches and according to~(\ref{eq:xN_def_deadBeat}) and~(\ref{eq:xN_def_quadr}) for the respective terminal set free approaches.
 Consider $\hat{J}_N^*(x_0;\delta)$ for a given $x_0 \in \mathcal{X}_N^\circ$ and let $\tilde{J}_N^*(x_0)$ and $J_N^*(x_0)$ denote to the value functions for the corresponding nonrelaxed problem formulation and a conventional MPC formulation, both based on the constraints that define the set $\mathcal{X}_N$. 
By definition, there exists for any $x_0 \in \mathcal{X}_N^\circ$  an feasible state and input sequences that result in strict satisfaction of all inequality constraints in the respective description of $\mathcal{X}_N$,
which implies that there exists a ${\delta}_0(x_0)\in \mathbb{R}_{\pplus}$ such that $\hat{J}_N^*(x_0;\delta)=\tilde{J}_N^*(x_0)$ $\forall \delta \leq \delta_0(x_0)$, see the proof of Lemma~\ref{lem:XbetaProp}. In particular, for any compact set of initial conditions $\mathcal{X}_0 \subseteq \mathcal{X}_N^\circ$ there exists a $\bar{\delta}_0 \in \mathbb{R}_{\pplus}$ such that for any $\delta \leq \delta_0$, any $x(0) \in \mathcal{X}_0$, and any $k\geq0$ it holds that $\hat{J}_N^*(x(k);\delta)=\tilde{J}_N^*(x(k))$. This shows that the performance of a corresponding nonrelaxed formulation can always be recovered within the interior of the respective feasible set. Furthermore, it is well known that the solution of the nonrelaxed barrier function based problem~(\ref{eq:OptProblemBarrier}) converges to the solution of the corresponding conventional problem when the barrier function weighting parameter~$\varepsilon$ approaches zero~\cite[ch.~11]{boydBook}. Thus, for arbitrary small $\varepsilon, \delta \in \mathbb{R}_{\pplus}$, the presented relaxed barrier function based MPC schemes recover the closed-loop behavior and performance of related conventional MPC schemes. 
The question whether and under which circumstances the relaxed barrier function based setup may also lead to an overall improved performance of the closed-loop system, e.g., a decreased cumulated cost, can be considered as possible future work.
To the experience of the authors, the resulting closed-loop behavior is  already very good for moderate values of the barrier parameters, i.e. $\varepsilon, \delta$ in the order of~$10^{-2}$, and typically close to the solution of the original MPC problem for $\varepsilon, \delta$ in the order of~$10^{-4}$. \\[0.2cm]
Concerning the robustness of the closed-loop system, it has to be noted that, due to the relaxation of the underlying state and input constraints, the resulting overall cost function is defined for any $x \in \mathbb{R}^n$. 
Hence, the presented MPC schemes are robust against arbitrary effects caused by disturbances, uncertainties, or measurement errors in the sense that the corresponding open-loop optimal control problem always admits a feasible solution.
However, can we say something about the qualitative robust stability properties of the closed-loop system? For example, we might consider the disturbance affected dynamics
 \begin{equation} \label{eq:distSystem}
  x(k+1)=Ax(k)+B\hat{u}_0^*(x(k))+w(k) \, ,
 \end{equation} 
 where $w(k)\in \mathbb{R}^n$ denotes an unknown but bounded additive disturbance at time instant $k\geq0$.
When considering the approach based on Theorem~\ref{thm:exactRel}, any stability or convergence guarantees of the closed-loop system will in general be lost as the disturbance might cause the system state to leave the set~$\hat{\mathcal{X}}_N(\delta)$.
On the other hand, the approaches based on Theorems~\ref{thm:globalStabilization},~\ref{thm:deadBeatStab}, and~\ref{thm:quadrBoundStab} allow to guarantee global asymptotic stability of the undisturbed closed-loop system. Interestingly, based on the main result of~\cite{angeli99}, this directly implies that system~(\ref{eq:distSystem}) is integral input-to-state stable~(iISS) with respect to the disturbance $w$.
 As an important consequence, 
the state remains bounded and converges asymptotically to the origin for any disturbance sequence with bounded energy, see~\cite{angeli99} for more details.
Further investigations, e.g., proving stronger input-to-state stability results for system~(\ref{eq:distSystem}), are the subject of possible future work. 
\section{Example and Numerical Aspects} \label{sec:relBarrierNumExample}
In this section, we outline a selection of steps which can be used to get from a given control problem to a stabilizing relaxed logarithmic barrier function based MPC scheme, and illustrate both the design and the behavior of the closed-loop system by means of a numerical example. 
In addition, we also briefly comment on some interesting numerical aspects of the resulting open-loop optimal control problem.
\subsection{Overall MPC Design Procedure}
\noindent\textbf{Step 1: Basic problem setup.} \ Choose suitable values for the problem parameters that are not related to the barrier functions, i.e., the weighting matrices $Q \in \mathbb{S}^n_{+}$ and $R \in \mathbb{S}^n_{\pplus}$ as well as the prediction horizon $N \in \mathbb{N}_+$. \\[0.1cm]
\textbf{Step 2: Relaxed barrier functions.} \ Decide on procedures for relaxing and recentering the logarithmic barrier functions and choose suitable (initial) values for the barrier function parameters $\varepsilon\!\in\!\mathbb{R}_{ \pplus}$, $\delta\!\in\!\mathbb{R}_{\pplus}$. In general, the quadratic relaxation $\beta(\cdot;\delta)=\beta_2(\cdot;\delta)$ seems to work well in practice. \\[0.1cm]
\textbf{Step 3: Terminal cost function.} \ Choose a suitable approach that allows to design the terminal cost $\hat{F}(\cdot)$ in such a way that asymptotic stability of the closed-loop system is guaranteed. The different approaches presented in this paper are summarized in Table~\ref{tab:approaches} together with additional information regarding the necessity of a terminal set $\mathcal{X}_f$, the underlying assumptions, and the respective region of attraction~(ROA). \\[0.1cm]
\textbf{Step 4: Parameter tuning.} Formulate the open-loop optimal control problem~(\ref{eq:OptProblemRelBarrier}) based on Steps~1\--3. The relevant problem parameters may be adjusted in order to achieve a desired closed-loop performance, e.g. based on closed-loop simulations. In addition to the usual MPC parameters, the parameters $\varepsilon$ and $\delta$ as well as the chosen terminal cost function will have a major impact on the resulting behavior.
To the experience of the authors, and in accordance with the above results, the barrier function weighting parameter $\varepsilon$ may be used to influence the closed-loop performance while the relaxation parameter~$\delta$ primarily allows to control the occurrence and the maximal amount of state and input constraint violations. 
In particular, Algorithm~1 from Section~\ref{sec:relBarrierConstr} may be used to adapt the relaxation parameter in such a way that satisfaction of a given constraint violation tolerance can be guaranteed. 
\begin{table*}
\caption{Summary of the presented MPC schemes based on relaxed logarithmic barrier functions.}
\label{tab:approaches}
 \centering
  \resizebox{\textwidth}{!}{
\begin{tabular}{lcrcclccc}
\toprule
Approach & Section  & Terminal cost $\hat{F}(x)$& $\mathcal{X}_f$   & ($A,B)$ & Relaxing function & Theorem & ROA  & constraint violation\\
\midrule
Relaxed $\mathcal{X}_f$  & \ref{sec:relBarrierStab1}1 & $x^\TRANSP P x + \varepsilon \hat{B}_{\mathrm{f}}(x)$  & yes  & stabilizable & $\beta_k(\cdot\,;\delta)/ \beta_e(\cdot\,;\delta)$ & Thm.~\ref{thm:exactRel} & $\hat{\mathcal{X}}_N(\delta)$ & none \\[0.1cm] 
Nonrelaxed $\mathcal{X}_f$ & \ref{sec:relBarrierStab1}2 & $x^\TRANSP P x + \varepsilon {B}_{\mathrm{f}}(x)$ & yes & controllable & $\beta_k(\cdot\,;\delta)/ \beta_e(\cdot\,;\delta)$  & Thm.~\ref{thm:globalStabilization}& $\mathbb{R}^n$ & adjustable, none in $\hat{\mathcal{X}}'_N(\delta)$ \\[0.1cm] 
Tail sequences &  \ref{sec:relBarrierStab2}1 & $\sum_{l=0}^{{T}-1} \hat{\ell}(z_l(x),v_l(x))$ & no  & controllable & $\beta_k(\cdot\,;\delta)/ \beta_e(\cdot\,;\delta)$ & Thm.~\ref{thm:deadBeatStab}& $\mathbb{R}^n$  & adjustable, none in $\hat{\mathcal{X}}'_N(\delta)$  \\[0.1cm] 
Quadratic bound & \ref{sec:relBarrierStab2}2 & $x^\TRANSP\! P x$ & no & stabilizable & $\beta_2(\cdot\,;\delta)\, $, quadratic & Thm.~\ref{thm:quadrBoundStab} & $\mathbb{R}^n$  & adjustable, none in $\hat{\mathcal{X}}'_N(\delta)$\\
\bottomrule
\end{tabular}}
\vspace*{-0.15cm}
\end{table*}
\subsection{Numerical Example}
In the following, we briefly illustrate the outlined design procedure and the closed-loop behavior of the proposed MPC schemes by means of an academic numerical example. We consider a discrete-time double integrator system of the form
 \begin{equation}
x(k+1)=\begin{bmatrix}
1 & T_s\\
0 & 1
\end{bmatrix}x(k) + \begin{bmatrix}
T_s^2 \\ T_s
\end{bmatrix} u(k)\,, \ \ \ T_s=0.1\,.
\end{equation}
with the input and state constraint sets $\mathcal{U}=\{u \in \mathbb{R}: -2\leq u \leq 1\}$ and $\mathcal{X}=\{x \in \mathbb{R}^2: -2\leq x_1 \leq 3, |x_2| \leq 0.8\}$.\\
Following Step~1 of our design procedure, we choose the basic problem parameters to $N=10$, $Q=\text{diag}(1,\, 0.1)$, and $R=1$.
In a second step, we decide on using weight recentered logarithmic barrier functions with a quadratic relaxation based on $\beta(\cdot\,;\delta)=\beta_2(\cdot\,;\delta)$, see Definition~\ref{def:recBarrier} and Eq.~(\ref{eq:relFunHauser}). We fix the barrier function weighting parameter to $\varepsilon=10^{-2}$ and set up the different MPC schemes from Table~\ref{tab:approaches} for varying values of the relaxation parameter~$\delta$. For the terminal set based approaches from Section~\ref{sec:relBarrierStab1}, we employ a contractive polytopic terminal set as discussed in~\cite{feller14b}. \\
Inspired by our results above, we are in particular interested in the region of attraction of the locally stabilizing approach from Section~\ref{sec:relBarrierStab1}1 as well as in the regions with guaranteed strict or approximate constraint satisfaction when considering the globally stabilizing approaches from the Sections~\ref{sec:relBarrierStab1}2 and \ref{sec:relBarrierStab2}1/C2. Furthermore, we would like to compare the closed-loop behavior of the corresponding different MPC schemes and illustrate how the maximal constraint violation can be controlled by adjusting the relaxation parameter. \\ 
Fig.~\ref{fig:differentSets} depicts and compares some of the $\delta$-dependent sets that have been discussed in the results above, i.e., regions of initial conditions for which we can guarantee properties like asymptotic stability or satisfaction of input and state constraints. Note that the superscript in each case indicates the respective MPC approach by referring to the corresponding theorem number (we use extra indices $a$ and $b$ to distinguish between the two tail sequence based approaches).
While the sets $\hat{\mathcal{X}}^2_N(\delta)$ and $\hat{\mathcal{X}}'^{\#}_N(\delta)$ are given as sublevel sets of the value function~$\hat{J}_N^*(x;\delta)$, the sets $\mathcal{X}_0^{\#}$ are sublevel sets of the resulting maximal constraint violation based on Theorem~\ref{thm:maxConstrViolation}, both for a given $\delta=10^{-6}$.
In this case, the functions were evaluated over a fine grid and \textsc{Matlab}'s \texttt{contour} function was used for plotting.
Whereas the sets for which strict constraint satisfaction can be ensured may be very small, approximate constraint satisfaction with a meaningful tolerance of $\hat{z}_{\textrm{tol}}=10^{-3}$ is achieved for much larger regions of initial conditions. The approach based on a purely quadratic terminal cost $\hat{F}(x)=x^\TRANSP\! Px$ results in very small sets which is due to the conservative quadratic upper bound (see Lemma~\ref{lem:quadraticBound}), that causes $P$ to grow rather fast depending on the relation~$\varepsilon/\delta^2$.  \\
In Fig.~\ref{fig:closedLoop}, the behavior of the resulting closed-loop systems is illustrated for different initial conditions and a varying relaxation parameter~$\delta$. It can be seen how convergence to the origin is achieved even for infeasible initial conditions and how strict or approximate constraint satisfaction may be enforced by making $\delta$ sufficiently small. Note that, in order to get comparable closed-loop performance and execution times, we chose a doubled horizon of $N=20$ for the MPC schemes based on a dead-beat tail sequence and a purely quadratic terminal cost.  
\subsection{Numerical Aspects} \label{sec:relBarrierNumerics}
As outlined above, one of the main advantages of relaxed barrier function based MPC formulations is given by the fact that the stabilizing control input can be characterized as the minimum of a globally defined, continuously differentiable, and strongly convex cost function that is parametrized by the current system state.    
In fact, after elimination of the linear system dynamics, the open-loop optimal control problem can be formulated as the unconstrained minimization of a cost function of the form
\begin{equation}\label{eq:QPformulation}
 \hat{J}_N(U,x)= \frac{1}{2}U^\TRANSP\! HU+x^{\!\TRANSP}\! FU + x^{\!\TRANSP}\!Yx + \varepsilon \hat{B}_{\mathrm{xu}}(U,x)\, ,
\end{equation}
where $U:=\begin{bmatrix}u_0^\TRANSP & \cdots & u_{N-1}^\TRANSP\end{bmatrix} \in \mathbb{R}^{n_U}$, $n_{U}=Nm$, and $\hat{B}_{\mathrm{xu}}: \mathbb{R}^{Nm} \times \mathbb{R}^n\to \mathbb{R}_+$ is a positive definite, convex, and continuously differentiable relaxed logarithmic barrier function for polytopic constraints of the form $GU\leq w+Ex$.
The matrices $H \in \mathbb{S}^{n_{U}}_{\pplus}$, $F\in\mathbb{R}^{n\times n_{U}}$, $Y\in\mathbb{S}^{n}_+$, $G \in \mathbb{R}^{q \times n_{U}}$, $w\in \mathbb{R}^q$, and $E\in\mathbb{R}^{q \times n}$  can be constructed from~(\ref{eq:OptProblemRelBarrier}) and the corresponding constraints by means of simple matrix operations.
Note that  the matrices in the above QP formulation may be rather ill-conditioned when considering unstable systems and long prediction horizons. This issue can, however, be resolved by a suitable prestabilization, see \cite{rossiter98}. 
The then resulting mixed input and state constraints are not considered here in detail but can be handled by all discussed approaches after some minor modifications.
\begin{figure}[t]
  \centering
 \input{figs/tac_xNdeltaSets_061214.tex}
  \includegraphics[scale=0.75]{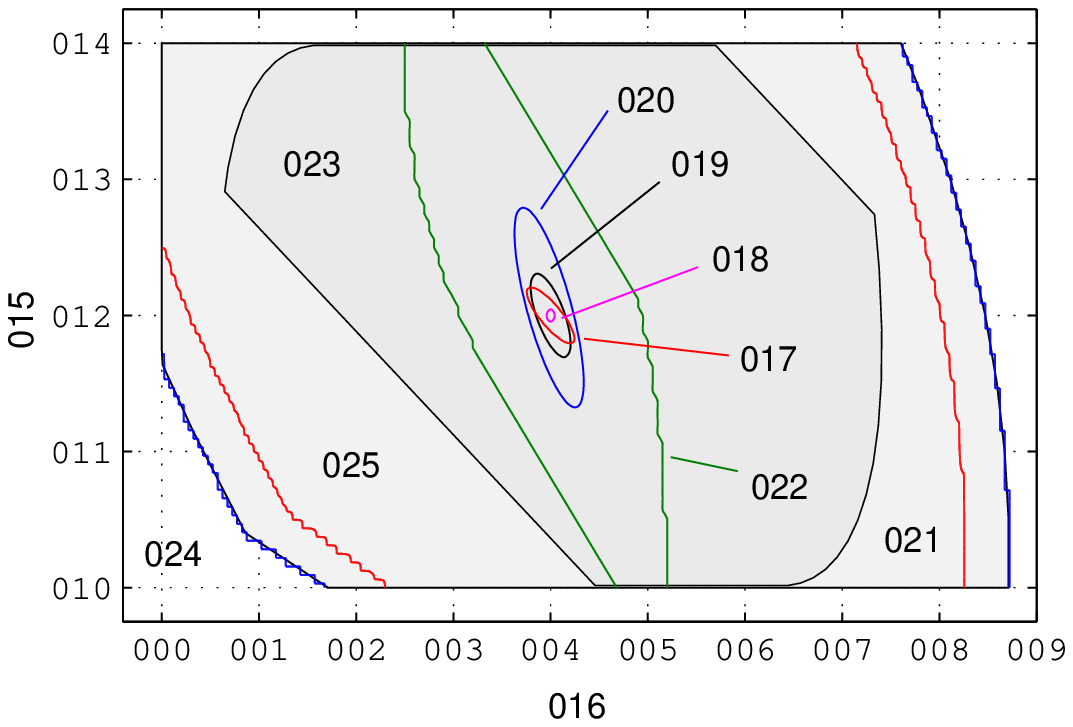}
    \vspace*{-0.5cm}\caption{The discussed sets for the different MPC schemes with $\varepsilon=10^{-2}$ and $\delta= 10^{-6}$. Here, $\mathcal{X}_f$ and $\mathcal{X}_N$ denote the polytopic terminal set and the resulting feasible set of the terminal set based approaches. Furthermore, $\hat{\mathcal{X}}^2_N(\delta)$ denotes the region of attraction of the locally stabilizing approach with strict constraint satisfaction, while $\hat{\mathcal{X}}'^3_N(\delta)$ and $\hat{\mathcal{X}}'^{4b}_N(\delta)$ 
  denote the sets with guaranteed zero constraint violation for the globally stabilizing approaches based on Theorems~\ref{thm:globalStabilization} and \ref{thm:deadBeatStab}, where the LQR based tail sequence is used in the latter case. 
   The corresponding sets $\hat{\mathcal{X}}'^{4a}_N(\delta)$ for a dead-beat controller based tail-sequence and $\hat{\mathcal{X}}'^{5}_N(\delta)$ for a purely quadratic terminal cost where in this case too small to be plotted.
In addition, the much larger sets $\mathcal{X}^3_0(\delta)$, $\mathcal{X}^{4a}_0(\delta)$, $\mathcal{X}^{4b}_0(\delta)$, and $\mathcal{X}^{5}_0(\delta)$ denote the regions in which the respective MPC schemes result in a maximal constraint violation of $\hat{z}_{\textrm{tol}}=10^{-3}$.} 
  \label{fig:differentSets}
 \vspace*{-0.25cm}
\end{figure}
 \begin{figure*}[t!]
  \begin{minipage}{0.32\textwidth}
 \centering
 \input{figs/tac_closedLoop_061214_N1020.tex}
  \includegraphics[scale=0.42]{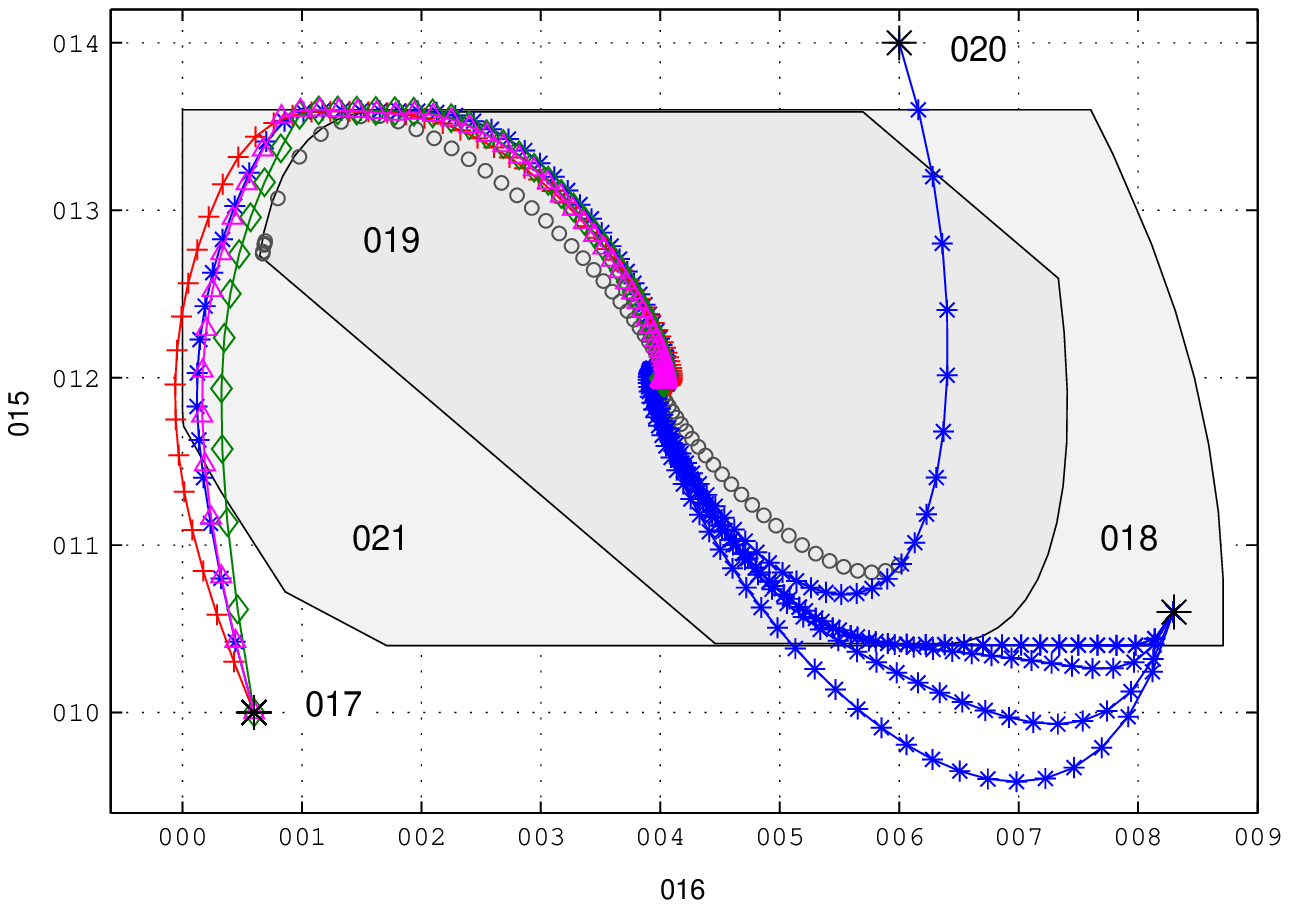}
  \end{minipage}
   \begin{minipage}{0.32\textwidth}
   \centering
 \input{figs/tac_closedLoop_inputs_061214_N1020.tex}
  \includegraphics[scale=0.42]{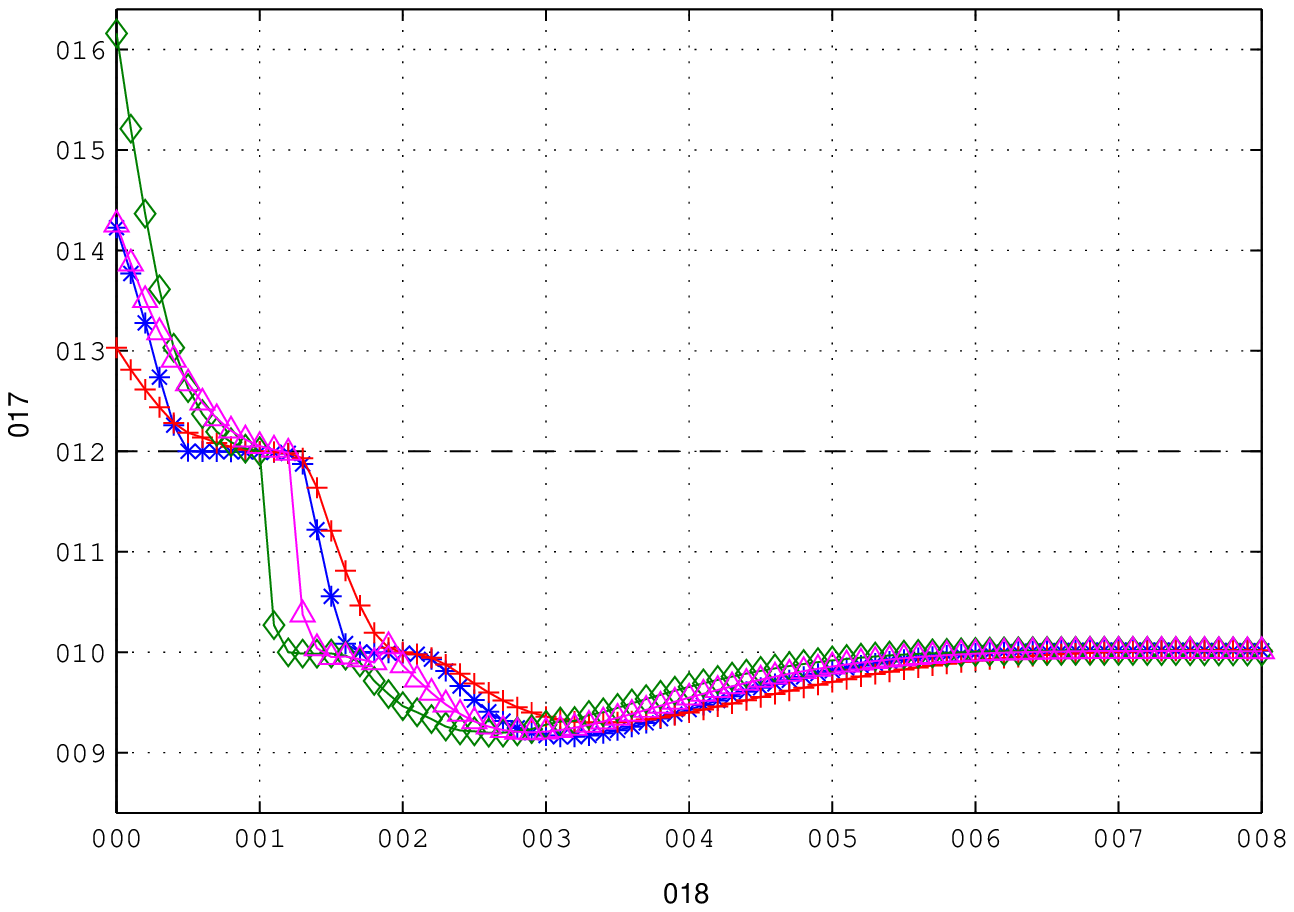}
  \end{minipage}
  \hspace*{0.1cm}
     \begin{minipage}{0.32\textwidth}
   \centering
 \input{figs/tac_deltaZero_081214_b.tex}
  \includegraphics[scale=0.42]{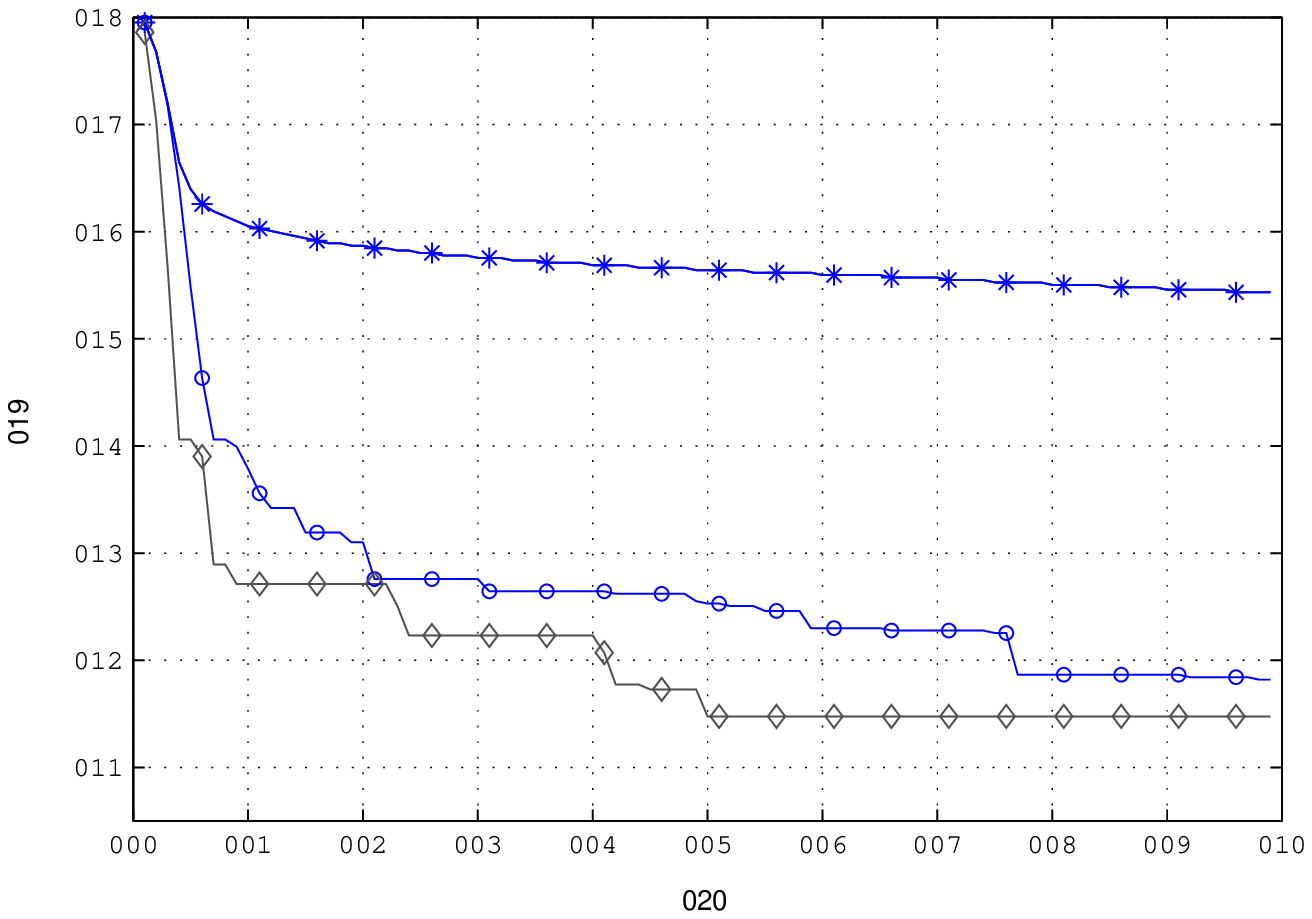}
  \end{minipage}
  \vspace*{0.1cm}
  \caption{\emph{Left:} closed-loop behavior for $x_{0,1}=[2.15,-0.7]^\TRANSP \in \mathcal{X}_N$ 
  and $\delta \in\{0.1,2.5\times10^{-2},10^{-2},10^{-3},10^{-4}\}$ for the approach based on a nonrelaxed terminal set~[{\textcolor{blue}{{$*$}}}]: if $\delta$ is small enough, the state trajectories stay feasible. Also shown is the behavior for two infeasible initial conditions $x_{0,2}$ and $x_{0,3}$ with $\delta=10^{-3}$ together with the corresponding predicted terminal states $x_N^*(x(k))$ [\textcolor{darkgreen}{$\circ$}]. For the infeasible initial condition $x_{0,2}=[-1.75,-1]^\TRANSP$, also the closed-loop trajectories of the globally stabilizing terminal set free MPC schemes from Section~\ref{sec:relBarrierStab2} are depicted (dead-beat based tail sequence~[$\textcolor{darkgreen}{\diamond}$], LQR based tail sequence~[{\footnotesize \textcolor{red}{$+$}}], purely quadratic terminal cost~[{\footnotesize \textcolor{magenta}{$\vartriangle$}}]). \newline
 \emph{Middle:} control input for the different globally stabilizing MPC schemes for the initial condition $x_{0,2}=[-1.75,-1]^\TRANSP$ using the same coloring scheme. \newline
 \emph{Right:} The relaxation parameter $\bar{\delta}_0$ obtained by Algorithm~1 for ensuring strict constraint satisfaction~[\textcolor{blue}{$\circ$}], i.e.~$\hat{z}_{\textrm{tol}}=0$, respectively approximate constraint satisfaction with a tolerance of $\hat{z}_{\textrm{tol}}=10^{-3}$~[\textcolor{blue}{$*$}] for the approach based on Theorem~\ref{thm:globalStabilization} with initial conditions in $\mathcal{X}_0=\kappa \mathcal{X}_N$ where $\kappa\in(0,1)$. Also depicted is the $\bar{\delta}_0$ required for ensuring asymptotic stability and strict constraint satisfaction directly based on the approach with a relaxed terminal set~[{\footnotesize \textcolor{darkgray}{$\diamond$}}], cf. Theorem~\ref{thm:exactRel}.}
  \label{fig:closedLoop}
  \vspace*{-0.15cm}
 \end{figure*}
While the above condensed QP formulation allows for a compact representation with a minimal number of optimization variables, the original formulation~(\ref{eq:OptProblemRelBarrier}) might allow to exploit the inherent problem structure, which may be beneficial from a numerical point of view. Tailored structure exploiting optimization algorithms have been presented in the context of both conventional linear MPC, e.g.\cite{rao98} and \cite{wang10}, and nonrelaxed barrier function based linear MPC, see~\cite{willsPHD}. Whether such techniques could also be applied for the relaxed barrier function based MPC approaches discussed in this paper, is an interesting open question which might be considered as possible future work.\\
Another interesting property of the condensed formulation~(\ref{eq:QPformulation}) is that it can be tackled directly by Newton-based continuous-time optimization algorithms as they are, for example, discussed in~\cite{feller13} and \cite{feller14}. 
In particular, the differentiability and convexity properties of the cost function allow to design a continuous-time dynamical system of the form
\begin{equation}
 \dot{U}(t)=f(U(t),x(t)), \quad U(t_0)=U_0
\end{equation} 
whose solution asymptotically tracks the optimal input vector~$\hat{U}^*(x(t))$, where $x(t)$ is the continuously measured system state, see the aforementioned references for more details. 
Such continuous-time linear MPC algorithms completely eliminate the need for an iterative on-line optimization and can be implemented for in principle arbitrary fast system dynamics. 
We expect that it might be possible to exploit some of the advantages of relaxed logarithmic barrier function based formulations, e.g., the global definition of the cost function and a bounded curvature, explicitly within the underlying numerical integration.
\section{Conclusion}
In summary, our investigation showed that the concept of relaxed logarithmic barrier function based model predictive control is interesting both from a system theoretical and a practical point of view.  
In particular, while we are still able to recover many of the theoretical properties of conventional or nonrelaxed barrier function based MPC schemes, the use of relaxed logarithmic barrier functions allows to characterize the stabilizing control input as the minimizer of a globally defined, continuously differentiable and strongly convex function that is parametrized by the current system state. 
As a main result, we presented different constructive MPC design approaches that guarantee global asymptotic stability and allow to influence the performance and constraint satisfaction properties of the closed-loop system directly by adjusting the relaxation parameter. The resulting MPC schemes are not necessarily based on the construction of a suitable terminal set and, due to the underlying relaxation, inherently robust against disturbances, uncertainties, or sensor faults. \\[0.2cm]
Interesting open problems may include a thorough analysis of the performance and robustness properties of the resulting closed-loop system, the derivation of less conservative constraint violation bounds, or the design and comparison of tailored iterative or continuous-time optimization algorithms that explicitly exploit the properties of the relaxed MPC problem formulation.    
\section*{Appendix}
%
In the following, we show how the finite-horizon LQR problem with zero terminal state constraint can be solved without the use of algebraic Riccati equations. Consider problem~(\ref{eq:zeroTstateLQR}) for a given initial state $x$ and horizon $T\geq n$. By eliminating the predicted states $z_l$ for $l=1\dots,T$ using~(\ref{eq:zeroTstateLQR}b), we get the following equivalent formulation in vectorized form:
\begin{align} \label{eq:zeroTstateLQR_QP}
J_V^*(x)&=\min_{V} \frac{1}{2}V^\TRANSP\! H V+\frac{1}{2}x^\TRANSP\!F V + \frac{1}{2}V^\TRANSP\!F^\TRANSP\! x + \frac{1}{2}x^\TRANSP\! Y x \nonumber \\
\mbox{s.\,t.} \ \ & A^Tx+\begin{bmatrix}S_1 & S_2 \end{bmatrix} V =0 \, ,
\end{align}
where the respective matrices are given as
 \begin{align*}
  S_1&=\begin{bmatrix}A^{T-1}B & \cdots & A^nB \end{bmatrix}, \ S_2=\begin{bmatrix}A^{n-1}B & \cdots & B \end{bmatrix} \\[0.1cm]
  H&=2\left(\tilde{R}+ \Phi^{\!\TRANSP}\! \tilde{Q} \Phi \right), \  F=2\Omega^{\!\TRANSP}\! \tilde{Q}\Phi, \ Y=2\left({Q}+ \Omega^{\!\TRANSP}\! \tilde{Q} \Omega \right) \\
  \Omega&=\begin{bmatrix}A \\[-0.1cm] \vdots \\ A^T \end{bmatrix}, \
  \Phi=\begin{bmatrix} B & \cdots & 0 \\[-0.1cm]
  \vdots & \ddots & \vdots \\
  A^{T-1}B  & \cdots & B
  \end{bmatrix}, \ \begin{array}{l} 
                  \tilde{Q}= I_T \otimes Q \\[0.3cm]
                   \tilde{R}=I_T \otimes R\,.
                \end{array}
 \end{align*}
Due to the controllability of the system, there always exists a feasible solution $V(x)$ to problem~(\ref{eq:zeroTstateLQR_QP}). We partition the vector $V \in \mathbb{R}^{Tm}$ as $V^\TRANSP=\begin{bmatrix} V_1^\TRANSP & V_2^\TRANSP \end{bmatrix}$ with $V_1 \in \mathbb{R}^{(T-n)m}$ and $V_2 \in \mathbb{R}^{nm}$. The terminal state constraint can be eliminated by choosing $V_2=-S_2^{+}\left(A^Tx+S_1V_1\right)$, where $S_2^+$ denotes the Moore-Penrose pseudoinverse of the matrix $S_2$. This leads to
\begin{equation}\label{eq:elimXNconstr}
V=\begin{bmatrix}I_{(T-n)m} \\ -S_2^{+}S_1 \end{bmatrix} V_1 + \begin{bmatrix} 0 \\ -S_2^{+}A^T \end{bmatrix}x = \Gamma_V V_1 + \Gamma_x x \ .
\end{equation}
Inserting~(\ref{eq:elimXNconstr}) into problem~(\ref{eq:zeroTstateLQR_QP}) results in the following unconstrained optimization problem in the variable $V_1$
\begin{equation}\label{eq:zeroTstateLQR_redQP}
J_V^*(x)=\min_{V_1} \frac{1}{2}V_1^\TRANSP\! \tilde{H} V_1+x^\TRANSP\!\tilde{F} V_1 + \frac{1}{2}x^\TRANSP\! \tilde{Y} x \ ,
 \end{equation} 
where the transformed problem matrices are given by $\tilde{H}=\Gamma_V^\TRANSP H\,\Gamma_V$, $\tilde{F}=\Gamma_x^\TRANSP  H\, \Gamma_V+F^\TRANSP\! \Gamma_V$, and $\tilde{Y}=Y+\Gamma_x^\TRANSP\! H\, \Gamma_x + \Gamma_x^\TRANSP\!F^\TRANSP+F\Gamma_x$.
Note that since $V(x)=\Gamma_x x$ is a feasible solution and $\mathrm{Im}(\Gamma_V)=\mathrm{Null}(\begin{bmatrix}S_1 & S_2\end{bmatrix})$, the solution is invariant under this change of coordinates, and the problems~(\ref{eq:zeroTstateLQR_QP}) and~(\ref{eq:zeroTstateLQR_redQP}) are equivalent, cf.~\cite[p.\,132]{boydBook}. Furthermore, it always
holds that $\tilde{H}\in\mathbb{S}^{(T-n)m}_{\pplus}$ as the matrix $\Gamma_V$ has full column rank. 
The optimal solution to the reduced problem can be computed easily as $V_1^*(x)=-\tilde{H}^{-1}\tilde{F}^\TRANSP x$ which results in 
\begin{equation}
V^*(x)=\left(\Gamma_x-\Gamma_V\tilde{H}^{-1}\tilde{F}^\TRANSP\right) x =K_V x\ .
\end{equation}
Furthermore, the optimal cost is given by $J_V^*(x)=x^\TRANSP\! P_Vx$ with $P_V=\frac{1}{2}\left(\tilde{Y}+\tilde{F}^\TRANSP \tilde{H}^{-1} \tilde{F} \right)\in \mathbb{S}^n_{+}$.
\bibliographystyle{plain}
\bibliography{ctrbMPC}
\end{document}